\documentclass[11pt,a4paper,reqno]{amsart}
\usepackage{eurosym}
\usepackage{graphicx}
\usepackage{amsfonts}
\usepackage{amsmath}
\usepackage{amssymb}
\usepackage[mathscr]{eucal}
\usepackage{exscale}
\usepackage{amscd}
\usepackage{tikz-cd}

\usepackage{geometry}
\newtheorem{theorem}{Theorem}[subsection]

\newtheorem{corollary}[theorem]{Corollary}

\newtheorem{definition}[theorem]{Definition}
\newtheorem{example}[theorem]{Example}

\newtheorem{lemma}[theorem]{Lemma}

\newtheorem{proposition}[theorem]{Proposition}
\newtheorem{remark}[theorem]{Remark}

\newcommand{\OX}[1]{\ensuremath \mathcal{O}_{#1}}			 
\newcommand{\OXss}[1]{\ensuremath \mathcal{O}_{\mathscr{#1}}} 

\newcommand{\Dolbn}{\ensuremath \mathcal{D}olb}				
\newcommand{\Dolb}[2]{\ensuremath \Dolbn(#1;#2)}	
\newcommand{\Dolbm}[2]{\ensuremath \Dolbn(#1;\mathcal{#2})}		
\newcommand{\Dolbss}[2]{\ensuremath \Dolbn(\mathscr{#1\,};\mathcal{#2})}  
\newcommand{\DolbA}[2]{\ensuremath \Dolbn(\mathcal{#1}\,;\mathcal{#2})}

\title[A Dolbeault-Grothendieck Resolution for Singular Spaces]
{A Dolbeault-Grothendieck Resolution for Singular Spaces}

\author{Andrei Baran}

\address{Institute of Mathematics of the Romanian Academy, P.O.Box 1-764, RO-014700 Bucharest, Romania}
\email{Andrei.Baran@imar.ro}

\date{\today}



\begin{document}

\begin{abstract} 
We construct a generalization of the Dolbeault-Grothendieck resolution on a singular 
complex space. The same construction yields, for each morphism of analytic spaces, a 
pullback mapping between the respective Dolbeault-Grothendieck resolutions.
As in the smooth case, the terms of the resolution are soft sheaves with stalks 
which are flat with respect to the sheaf of holomorphic sections. If, moreover, the 
complex space $(X,\OX{X})$ is countable at infinity then the global section spaces of 
the terms of the resolution are endowed with natural Fr\'{e}chet-Schwarz topologies 
which induce the natural topology on the cohomology groups $H^{\bullet }(X,\OX{X})$.
The construction is an exercise in globalization using semi-simplicial techniques.
Using the above construction one can produce, for instance, a soft resolution with 
$\OX{X}$-flat stalks for the de Rham complex on the analytic space $X$.
\end{abstract}

\maketitle

\section{Introduction}

\refstepcounter{subsection} Let $X$ be an n-dimensional complex
manifold and let 
\begin{equation}
0\longrightarrow \OX{X} \longrightarrow \mathcal{E}_{X}^{0,0}\overset%
{\overline{\partial }}{\longrightarrow }...\overset{\overline{\partial }}{%
\longrightarrow }\mathcal{E}_{X}^{0,n}\longrightarrow 0  \label{Rezol_Dolb1}
\end{equation}%
be the Dolbeault-Grothendieck resolution on $X$. Here, as usual, $\OX{X}$ is 
the sheaf of holomorphic functions and $\mathcal{E}_{X}^{p,q}$ the
sheaf of $(p,q)$ smooth differential forms on $X$. The problem is that in
the singular case the complex (\ref{Rezol_Dolb1}) is no longer a resolution
for $\OX{X}$. The purpose of this paper is to construct an analogue
for the Dolbeault-Grothendieck resolution on a complex space with
singularities.

There exist two recent constructions of analogues of Dolbeault-Grothendieck
resolutions under suplimentary hypothesis on the singular space. Ancona and
Gaveau \cite{Anc-Gav} considered analytic spaces with smooth singular locus;
their solution is based on Hironaka desingularization. Andersson and
Samuelsson \cite{And-Samuel} considered the case of a reduced analytic
space; the resolution is obtained as a subcomplex of the complex of smooth
currents on the space; their construction uses Koppelman representation
formulas.

Our construction is based on working in a category larger than that of
analytic spaces.--- the category of semi-simplicial analytic spaces. Recall
that a s.s.analytic space (throughout this paper s.s.is short for
semi-simplicial) is a contravariant functor from a simplicial complex (seen
as a category) to the category of analytic spaces, or, equivalently, a
family of analytic spaces indexed by the simplexes of a simplicial complex,
together with a family of compatible connecting morphisms (see Section \ref%
{Sect_ss} for the definitions). S.s.analytic spaces and the corresponding
analytic modules proved a very flexible tool. They appeared implicitely or
explicitely, for instance, in Forster, Knorr \cite{ForstK} for the proof of
Grauert's direct image theorem, in Verdier \cite{Verd}, Baran \cite{B1} for
the introduction of natural topologies on the global (hyper)cohomological
invariants of analytic sheaves, in Ramis, Ruget \cite{R-R} for the proof of
relative analytic duality, in Flenner \cite{Flenner} and B\u{a}nic\u{a},
Putinar, Schumacher \cite{BPS} for computations linked to deformation theory.

Our construction solves the problem for any analytic space. In fact for each
pair $(X,\mathcal{A})$, where $\mathcal{A}$\ is an embedding atlas of the
analytic space $X$ (i.e. a family of local closed embeddings of $X$ in
complex manifolds - see paragraph \ref{paragr_Atlas}) we produce a
resolution for $\OX{X}$, denoted by $\Dolb{\mathcal{A}\,}{\OX{X}}$. The pair 
$(X,\mathcal{A})$ will be called a \textit{locally embedded analytic space}. 
In particular, if $X$ is a complex manifold and $\mathcal{A}$ the obvious atlas with 
one chart, then one gets the usual resolution on $X$. For each morphism of locally embedded 
analytic spaces: 
\begin{equation*}
f:(X,\mathcal{A})\rightarrow (Y,\mathcal{B})
\end{equation*}%
one constructs a pullback morphism which extends the pullback morphism from
the smooth case: 
\begin{equation}
f^{\ast }:\Dolb{\mathcal{B}\,}{\OX{Y}}\rightarrow 
f_{\ast }\Dolb{\mathcal{A}\,}{\OX{X}}
\label{pull_back_introd}
\end{equation}

Moreover, the same construction produces a resolution for each $\OX{X}$-module 
$\mathcal{F}$, denoted $\DolbA{A}{F}$.

The resolution $\DolbA{A}{F}$ depends on the embedding atlas $\mathcal{A}$. However 
the resolution is unique up to unique isomorphism in the derived category of 
$\OX{X}$-modules, $D(\OX{X})$. More precisely, if $(X,\mathcal{A}),(Y,\mathcal{B})$ 
are locally embedded analytic spaces and $f:X\rightarrow Y$ is a morphism of
analytic spaces (but not necessarily of locally embedded analytic spaces)
then there exists in $D(\OX{Y})$\ a unique pullback morphism similar to 
(\ref{pull_back_introd}) (see Theorem \ref{Theor_Dolb_cat_deriv}). In particular, 
if $\mathcal{A}$ and $\mathcal{B}$ are two embedding atlases on the same analytic space 
$X$ then there is a unique isomorphism between $\DolbA{A}{F}$ and $\DolbA{B}{F}$ 
in $D(\OX{X})$.

The main result of the paper is:

\begin{theorem}
\label{Theor_Dolb}

\begin{enumerate}
\item Let $(X,\mathcal{A})$ be a locally embedded analytic space and $%
\mathcal{F}\in Mod(\OX{X})$. Then there is a functor 
\begin{equation}
\Dolb{\mathcal{A}\,}{\bullet}:Mod(\OX{X})\rightarrow C^{+}(X)
\end{equation}
such that:

\begin{enumerate}
\item $\Dolb{\mathcal{A}\,}{\bullet}$ is an exact functor

\item There is a functorial morphism $\mathcal{F}\rightarrow \DolbA{A}{F}$ 
and $\DolbA{A}{F}$ is a resolution of $\mathcal{F}$.

\item $\DolbA{A}{F}$ has soft components.

\item $\Dolb{\mathcal{A}\,}{\OX{X}}$ has $\OX{X}$-flat components

\item One has a natural quasi-isomorphism: 
\begin{equation}
\Dolb{\mathcal{A}\,}{\OX{X}}\otimes _{\OX{X}}%
\mathcal{F}\rightarrow \DolbA{A}{F}	
\label{ident_Dolb_A}
\end{equation}%
Moreover, if $\mathcal{F}\in Coh(\OX{X})$\ then the above morphism
is an isomorphism.

\item If $X$ is a complex manifold and $\mathcal{A}$ consists of only one
chart, namely $(X,id,X)$, then $\Dolb{\mathcal{A}\,}{\bullet}$ coincides 
with the usual Dolbeault-Grothendieck resolution on $X$.
\end{enumerate}

\item Let $f:(X,\mathcal{A})\rightarrow (Y,\mathcal{B})$ be a morphism of
locally embedded analytic spaces,\linebreak $\mathcal{F}\in Mod(\OX{X})$, $%
\mathcal{G}\in Mod(\OX{Y})$ and $u:\mathcal{G}\rightarrow f_{\ast }%
\mathcal{F}$ a morphism of $\OX{Y}$-modules. Then there exists a
natural pullback morphism: 
\begin{equation}
f^{\ast }(u):\DolbA{B}{G}\rightarrow f_{\ast }\DolbA{A}{F}  
\label{def_f*_u}
\end{equation}%
such that the following diagram commutes:%
\begin{equation}
\begin{CD} \DolbA{B}{G} @>f^{\ast }(u)>> f_{\ast
}\DolbA{A}{F} \\ @AAb^{\prime}A @AAf_{*}bA\\
\mathcal{G} @>{u}>> f_{\ast }\mathcal{F} 
\end{CD}  \label{diagr_f*_u}
\end{equation}%
In particular there is a natural morphism:%
\begin{equation}
f^{\ast }:\Dolb{\mathcal{B}\,}{\OX{Y}}\rightarrow %
f_{\ast }\Dolb{\mathcal{A}\,}{\OX{X}}
\label{def_f*}
\end{equation}%
over the mapping $f^{\ast }:\OX{Y} \rightarrow f_{\ast }\OX{X}$.

\item Let $(X,\mathcal{A})\overset{f}{\rightarrow }(Y,\mathcal{B})\overset{g}%
{\rightarrow }(Z,\mathcal{C})$ be morphisms of locally embedded analytic
spaces and $h=g\circ f$. Let moreover $\mathcal{F}\in Mod(\OX{X})$, 
$\mathcal{G}\in Mod(\OX{Y})$, $\mathcal{H}\in Mod(\OX{Z})$
and morphisms $u:\mathcal{G}\rightarrow f_{\ast }\mathcal{F}$ 
$\OX{Y}$-linear, $v:\mathcal{H}\rightarrow g_{\ast }\mathcal{G}$, 
$w:\mathcal{H} \rightarrow h_{\ast }\mathcal{F}$ $\OX{Z}$-linear, such that 
$g_{\ast }(u)\circ v=w$ then one has the commutative diagram:
\begin{equation}
\begin{tikzcd}[column sep=-0.2cm] h_{\ast}\DolbA{A}{F} && g_{\ast}\DolbA{B}{G}
\arrow{ll}[swap]{g_{\ast}(f^{\ast}(u)} \\ & \DolbA{C}{H} 
\arrow{ul}{h^{\ast}(w)} \arrow{ur}[swap]{g^{\ast}(v)}
\end{tikzcd}  \label{diagr_3_A}
\end{equation}
\end{enumerate}
\end{theorem}

\medskip

The proof is based on two simple remarks:

\begin{enumerate}
\item \label{rem_functorialit_pullback} Let $\mathscr{X}=((X_{\alpha
})_{\alpha \in \mathcal{S}})$ be a s.s.complex manifold relative to the
simplicial complex $(I,\mathcal{S})$. The compatibility of the pullback of
differential forms with the composition of mappings ensures that the
Dolbeault-Grothendieck resolutions on the manifolds $X_{\alpha }$ form a
complex of $\mathscr{X}$-modules. We denote it $\Dolb{\mathscr{X}\,}{\OXss{X}}$. 
Moreover, if 
\begin{equation}
F:\mathscr{X}\rightarrow \mathscr{Y}
\end{equation}%
is a morphism of s.s.complex manifolds (see Definition \ref{Def_morph_over_f})
then one defines a pullback morphism:%
\begin{equation}
F^{\sharp}:\Dolb{\mathscr{Y}\,}{\OXss{Y}} \rightarrow
F_{\sharp} \Dolb{\mathscr{X}\,}{\OXss{X}}
\label{pullback_manif}
\end{equation}%
where $F_{\sharp}$ is a variant of the direct image functor which associates
to each $\mathscr{X}$-module a complex of $\mathscr{Y}$-modules.

\item \label{rem_emb_an_sp} Let $X\overset{i}{\hookrightarrow }D$ be an
analytic subspace of the complex manifold $D$, given by the coherent ideal $%
\mathcal{I}\subset \OX{D}$ (we say that $(X,i,D)$ is an embedding
triple - see paragraph \ref{paragr_emb_triples}). The complex obtained by
tensoring the Dolbeault-Grothendieck resolution on $D$ with $\OX{D}/%
\mathcal{I}$ (which comes to restricting the coefficients of the
differential forms on $D$ to $X$) is a resolution of $\OX{X}$,
since the $\OX{D}$-modules $\mathcal{E}_{D}^{p,q}$ are $\OX{D}$-flat 
(see Malgrange \cite{M}). We consider this complex as the analogue for the 
Dolbeault-Grothendieck resolution on $X$. Note that the complex described here 
appears in the proof of the duality theorems of Serre-Malgrange (see Malgrange 
\cite{M1}\ or B\u{a}nic\u{a}, St\u{a}n\u{a}\c{s}ila \cite{B-S} Ch 7 \S 4.b). 
If $(\mathscr{X},k,\mathscr{D})$ is a s.s.embedding triple (see Remark 
\ref{Rem_ss_embed_triple}) then the Dolbeault-Grothendieck resolutions on each component 
form a complex of $\mathscr{X}$-modules that we denote by $\Dolb{k\,}{\OX{X}|\mathfrak{U}}$.
\end{enumerate}

Let $(X,\mathcal{A})$ be a locally embedded analytic space. By Lemma \ref%
{Lemma_assoc_ss_triples} and Example \ref{Ex_incluz_acop} one associates to $%
(X,\mathcal{A})$ a s.s.embedding triple $(\mathfrak{U},k,\mathfrak{D})$ and
a natural morphism of s.s.analytic spaces $b:\mathfrak{U}\rightarrow X$.
According to \textbf{2.} there is a $\Dolbn$-resolution on $\mathfrak{U}$. 
We need to define a $\Dolbn$-resolution on $X$, 
$\Dolb{\mathcal{A}\,}{\OX{X}}$, such that a pullback mapping
similar to (\ref{pullback_manif}) exist for $b$, i.e. a mapping :%
\begin{equation*}
b^{\sharp}:\Dolb{\mathcal{A}\,}{\OX{X}}\rightarrow b_{\sharp }%
\Dolb{k\,}{\OX{X}|\mathfrak{U}}
\end{equation*}%
For this we simply set:%
\begin{equation*}
\Dolb{\mathcal{A}\,}{\OX{X}}=b_{\sharp }\Dolb{k\,}
{\OX{X}|\mathfrak{U}}\text{ and }b^{\sharp}=id
\end{equation*}%
and check that all the properties are verified.

Here are some applications of Theorem \ref{Theor_Dolb}.

Since the terms of the Dolbeault-Grothendieck resolution are soft sheaves
one can use it to define representatives for derived functors and morphisms.
In particular the complex $\Gamma (X,\DolbA{A}{F})$ computes the cohomology of 
$X$ with coefficients in $\mathcal{F}$; furthermore, if $\mathcal{F}$ is a coherent 
sheaf then the terms of $\Gamma(X,\DolbA{A}{F})$ are endowed with Fr\'{e}%
chet-Schwarz topologies which induce the natural topologies on the
cohomology groups of $\mathcal{F}$(see Corollary \ref{Corol_topol_FS}). Note
that since for each open covering of the analytic space one produces a
resolution and the construction has good functorial properties, it follows
that the resolution is suitable to produce good representatives for derived
functors and morphisms.

If $X$ is a reduced analytic space then, by using a direct limit argument,
one can construct on $X$ an analogue of the Dolbeault-Grothendieck
resolution which coincides with the classical one on $Reg(X)$, the regular
locus of $X$. However, in this case the topologies on the global sections of
the resolution are more complicated.

One can link the complex $\Dolb{\mathcal{A}\,}{\OX{X}}$ to the complex of smooth 
differential forms on $X$, namely there is a natural surjective morphism between 
$\Dolb{\mathcal{A}\,}{\OX{X}}$ and a suitable \v{C}ech complex of the complex of 
smooth differential forms on $X$(see theorem \ref{Theor_Dolb_diff_forms}).

As in the smooth case, by using the $\Dolb{\mathcal{A}\,}{\bullet }$
-functor one can construct a resolution with soft sheaves for the de Rham
complex on an analytic space $X$(see Theorem \ref{Theor_Dolb_deRham}).

The main result of this note was announced in \cite{B2}. In the same paper
the functor $F_{\sharp}$ (denoted there by $F_{\ast }$) is defined.

In a future paper, using roughly the same technique as here, but replacing
the functor $F_{\sharp }$ with the direct image with proper supports we
shall give a construction of the dualizing complex of an analytic space.

\section{Preliminaries \label{Sect_prelim}}

\textbf{Review and notations.} Throughout this paper analytic space will
mean complex analytic space.

Let $(X,\OX{X})$ be an analytic space. We use the following notations:

\begin{itemize}
\item[-] $Mod(\OX{X})$ - the abelian category of $\OX{X}$-modules; $Coh(\OX{X})$ - the subcategory of coherent $\OX{X}$-modules

\item[-] $C(X)$ the abelian category of complexes of $\OX{X}$%-modules; $D(X)$ the derived category of $Mod(\OX{X})$

\item[-] As usual, $C^{\ast }(X)$, respectively $D^{\ast }(X)$, where $\ast
=+,-,b$, denote the subcategories of complexes bounded below, bounded above,
respectively bounded
\end{itemize}

Let $X$ be an $n$-dimensional complex manifold. We denote by $\mathcal{E}%
_{X}^{p,q}$ the sheaf of $(p.q)$-differential forms with $C^{\infty }$
coefficients on $X$. It is a soft sheaf and, according to \cite{M}, it is an 
$\OX{X}$-flat module. The complex of $\OX{X}$-modules: 
\begin{equation}
0\longrightarrow \mathcal{E}_{X}^{0,0}\overset{\overline{\partial }}{%
\longrightarrow }...\overset{\overline{\partial }}{\longrightarrow }\mathcal{%
E}_{X}^{0,n}\longrightarrow 0  \label{Rezol_Dolb}
\end{equation}%
is the Dolbeault-Grothendieck resolution of $\OX{X}$.

For $f:X\rightarrow Y$ a holomorphic mapping between two complex manifolds
we denote \ $f^{\ast }:\mathcal{E}_{Y}^{p,q}\rightarrow f_{\ast }\mathcal{E}%
_{X}^{p,q}$ the $\OX{Y}$-linear morphism given by the pullback of forms.

It is well known that if $X\overset{f}{\rightarrow }Y\overset{g}{\rightarrow 
}Z$ are holomorphic mappings between complex manifolds and $h=g\circ f$, then
one has the commutative diagram:\smallskip 
\begin{equation}
\begin{tikzcd}[column sep=0.1cm] 
h_{\ast}\mathcal{E}_{X}^{p.q} 
&& g_{\ast}\mathcal{E}_{Y}^{p.q}
\arrow{ll}[swap]{g_{\ast}(f^{\ast})} \\ 
& \mathcal{E}_{Z}^{p.q} 
\arrow{ul}{h^{\ast}} \arrow{ur}[swap]{g^{\ast}}
\end{tikzcd}
\label{Diagr_3_E}
\end{equation}

\section{Semi-simplicial Objects \label{Sect_ss}}

\refstepcounter{subsection}\textbf{\arabic{section}.%
\arabic{subsection}} \label{paragr_ss_an_sp}\textbf{Semi-simplicial analytic
spaces. }Let $(I,\mathcal{S})$ be a simplicial complex, i.e. $I$ is a set
and $\mathcal{S}$ is a family of non-empty finite parts of $I$, called
simplexes, such that:

\begin{enumerate}
\item $\{i\}\in \mathcal{S}$ for all $i\in I$

\item if $\alpha ^{\prime }\subset \alpha \in \mathcal{S}$ then $\alpha
^{\prime }\in \mathcal{S}$
\end{enumerate}

If $\alpha \in \mathcal{S}$ \ we denote by $|\alpha |\,=Card(\alpha )-1$ the
length of the simplex $\alpha $. Recall that $\dim ((I,\mathcal{S}))=\sup
\{|\alpha |\ |\ \alpha \in \mathcal{S}\}$.

A morphism of simplicial complexes $f:(I,\mathcal{S})\rightarrow $ $(J,%
\mathcal{T})$ is simply a mapping $f:I\rightarrow J$ such that $f(\alpha
)\in \mathcal{T}$ whenever $\alpha \in \mathcal{S}$. If $K(pt)$ is the
simplicial complex over the set with one element $\{pt\},$ then we denote by
$a_{\mathcal{S}}:(I,\mathcal{S})\rightarrow K(pt)$ the morphism
induced by the unique mapping $I\rightarrow \{pt\}$.

\begin{definition}
\begin{enumerate}
\item Let $\mathcal{C}$ be a category. A semi-simplicial (s.s.) system of
objects in $\mathcal{C}$ indexed by the simplicial complex $(I,\mathcal{S})$
consists of:

\begin{itemize}
\item[-] a family $(X_{\alpha })_{\alpha \in \mathcal{S}}$ of objects in $%
\mathcal{C}$

\item[-] a family $(\rho _{\alpha \beta })_{\alpha \subset \beta }$ of
connecting morphisms, $\rho _{\alpha \beta }:X_{\beta }\rightarrow X_{\alpha
}$, such that \linebreak 
$\rho _{\alpha \alpha }=id$ for $\alpha \in \mathcal{S}$, and $%
\rho _{\alpha \beta }\circ \rho _{\beta \gamma }=\rho _{\alpha \gamma }$
whenever $\alpha \subseteq \beta \subseteq \gamma $.
\end{itemize}

\item Let $\mathscr{X}=((X_{\alpha })_{\alpha \in \mathcal{S}}$, $(\rho
_{\alpha \beta })_{\alpha \subset \beta })$, $\mathscr{Y}=((Y_{\alpha
})_{\alpha \in \mathcal{S}}$, $(\rho _{\alpha \beta }^{\prime })_{\alpha
\subset \beta })$ be s.s.systems of objects in $\mathcal{C}$ indexed by $(I,%
\mathcal{S})$. A morphism $F:\mathscr{X}\rightarrow \mathscr{Y}$ consists of
a family of morphisms in $\mathcal{C}$, $(F_{\alpha })_{\alpha \in \mathcal{S%
}}$, $F_{\alpha }:X_{\alpha }\rightarrow Y_{\alpha }$, such
that $F_{\alpha }\circ \rho _{\alpha \beta }=\rho _{\alpha \beta
}^{\prime }\circ F_{\beta }$.
\end{enumerate}
\end{definition}

If the simplicial complex is clear from the context we shall omit mentioning
it.

If $\mathcal{C}$ is the category of analytic spaces then we say for short
s.s.analytic space instead of s.s.system of analytic spaces. Let $\mathscr{X}%
=((X_{\alpha })_{\alpha \in \mathcal{S}}$, $(\rho _{\alpha \beta })_{\alpha
\subset \beta })$ be a s.s.analytic space. Here $X_{\alpha }$ is short for $%
(X_{\alpha },\OX{\alpha })$, where $\OX{\alpha }$ denotes the sheaf of holomorphic 
sections of $X_{\alpha }$, and $\rho _{\alpha \beta}$ is short for 
$(\rho _{\alpha \beta },\rho _{\beta \alpha }^{1})$ where 
$\rho _{\alpha \beta }:X_{\beta }\rightarrow X_{\alpha }$ is the topological
part and $\rho _{\beta \alpha }^{1}:\OX{\alpha }\rightarrow \rho
_{\alpha \beta \ast }(\OX{\beta })$ is the sheaf level part. If $%
X_{\alpha }$ is a complex manifold for all $\alpha \in \mathcal{S}$ then $%
\mathscr{X}$ will be called a s.s.complex manifold.

\begin{remark}
\label{anal_sp_Kpt}An analytic space can be regarded as a s.s.analytic space
indexed by $K(pt)$, the simplicial complex constructed over the index set
with one element.
\end{remark}

\begin{example}
\label{ex_acop}Let $X$ be an analytic space and $\mathcal{U}=(U_{i})_{i\in
I} $ an open covering of $X$. One associates to $\mathcal{U}$

\begin{enumerate}
\item[-] the simplicial complex $(I,\mathcal{N(U)})$, where $\mathcal{N(U)}$
denotes the nerve of $\mathcal{U}$

\item[-] the s.s.analytic space indexed by $(I,\mathcal{N(U)})$, 
\begin{equation*}
\mathfrak{U}=((U_{\alpha })_{\alpha \in \mathcal{N(U)}},(i_{\alpha \beta
})_{\alpha \subset \beta })
\end{equation*}%
where $U_{\alpha }$ denotes, as usual, the intersection $\bigcap\limits_{i%
\in \alpha }U_{i}$, and $i_{\alpha \beta }:U_{\beta }\rightarrow U_{\alpha }$
is the natural inclusion.
\end{enumerate}
\end{example}

\begin{example}
\label{ex_prod}Let $(I,\mathcal{S})$ be a simplicial complex and $%
(X_{i})_{i\in I}$ a family of analytic spaces. For $\alpha \in \mathcal{S}$
let $X_{\alpha }=\prod\limits_{i\in \alpha }X_{i}$. Then $\mathscr{X}%
=((X_{\alpha })_{\alpha \in \mathcal{S}}$, $(p_{\alpha \beta })_{\alpha
\subset \beta })$ is a s.s.analytic space, where $p_{\alpha \beta }:X_{\beta
}\rightarrow X_{\alpha }$ is the natural projection.
\end{example}

\begin{definition}
\label{Def_morph_over_f}Let $\mathcal{C}$ be a category, $f:(I,\mathcal{S}%
)\rightarrow $ $(J,\mathcal{T})$ a morphism of simplicial complexes, $%
\mathscr{X}=((X_{\alpha })_{\alpha \in \mathcal{S}}$, $(\rho _{\alpha \beta
})_{\alpha \subset \beta })$, $\mathscr{Y}=((Y_{\gamma })_{\gamma \in 
\mathcal{T}}$, $(\rho _{\gamma \delta }^{\prime })_{\gamma \subset \delta })$
s.s.systems of objects in $\mathcal{C}$ indexed by $(I,\mathcal{S})$,
respectively $(J,\mathcal{T})$. A morphism $F:\mathscr{X}\rightarrow %
\mathscr{Y}$\ of s.s.systems of objects in $\mathcal{C}$ over $f$ consists
of a family of morphisms in $\mathcal{C}$, $(F_{\alpha })_{\alpha \in 
\mathcal{S}}$, $F_{\alpha }:X_{\alpha }\rightarrow Y_{f(\alpha )}$%
, such that\ $F_{\alpha }\circ \rho _{\alpha \beta }=\rho _{f(\alpha
)f(\beta )}^{\prime }\circ F_{\beta }$.
\end{definition}

\begin{example}
\label{Ex_incluz_acop}Let $X$ be an analytic space, $\mathcal{U}%
=(U_{i})_{i\in I}$ an open covering of $X,$ and $\mathfrak{U}=((U_{\alpha
})_{\alpha \in \mathcal{N(U)}},(i_{\alpha \beta })_{\alpha \subset \beta })$
the s.s.analytic space associated to $\mathcal{U}$ (see Example \ref{ex_acop}%
). Then the inclusion mappings $i_{\alpha }:U_{\alpha }\rightarrow X$
determine a morphism of s.s.analytic spaces \linebreak 
$i:$ $\mathfrak{U\rightarrow }X$ over 
$a_{\mathcal{N(U)}}:\mathcal{N(U)}\rightarrow K(pt)$.
\end{example}

\refstepcounter{subsection} \textbf{\arabic{section}.\arabic{subsection}} 
\textbf{Modules over s.s.analytic spaces. }Unless otherwise stated, in this
section $\mathscr{X}=((X_{\alpha },\OX{\alpha })_{\alpha \in 
\mathcal{S}},(\rho _{\alpha \beta },\rho _{\beta \alpha }^{1})_{\alpha
\subset \beta })$ will denote a s.s.analytic space indexed by the simplicial
complex $(I,\mathcal{S})$.

\begin{definition}
\begin{enumerate}
\item An $\mathscr{X}$-module consists of

\begin{itemize}
\item[-] a family $(\mathcal{F}_{\alpha })_{\alpha \in \mathcal{S}}$ \ where 
$\mathcal{F}_{\alpha }$ is an $\OX{\alpha }$-module for each $\alpha \in 
\mathcal{S}$

\item[-] a family of connecting morphisms $(\varphi _{\beta \alpha
})_{\alpha \subset \beta }$, where%
\begin{equation*}
\varphi _{\beta \alpha }:\mathcal{F}_{\alpha }\rightarrow \rho _{\alpha
\beta \ast }(\mathcal{F}_{\beta })
\end{equation*}%
is a morphism of $\OX{\alpha }$-modules such that $\varphi _{\alpha
\alpha }=id$ for all $\alpha \in \mathcal{S}$, and \linebreak 
$\rho _{\beta \gamma \ast}(\varphi _{\gamma \beta })\circ 
\varphi _{\beta \alpha }=\varphi _{\gamma \alpha }$ whenever 
$\alpha \subseteq \beta \subseteq \gamma $.
\end{itemize}

\item If $\mathcal{F}=((\mathcal{F}_{\alpha })_{\alpha \in \mathcal{S}}\
,(\varphi _{\beta \alpha })_{\alpha \subset \beta })$, $\mathcal{G}=((%
\mathcal{G}_{\alpha })_{\alpha \in \mathcal{S}}\ ,(\psi _{\beta \alpha
})_{\alpha \subset \beta })$ are $\mathscr{X}$-modules, then a morphism of\
\ $\mathscr{X}$-modules \ $u:$\ $\mathcal{F}\rightarrow \mathcal{G}$
consists of a family $(u_{\alpha })_{\alpha \in \mathcal{S}}$, where $%
u_{\alpha }:\mathcal{F}_{\alpha }\rightarrow \mathcal{G}_{\alpha }$ is a
morphism of $\OX{\alpha }$-modules, such that for $\alpha \subseteq
\beta $ $\rho _{\alpha \beta \ast }(u_{\beta })\circ \varphi _{\beta \alpha
}=\psi _{\beta \alpha }\circ u_{\alpha }$.
\end{enumerate}
\end{definition}

We denote by $Mod(\mathscr{X})$ the abelian category of $\mathscr{X}$%
-modules and by $C(\mathscr{X})$ the category of complexes with terms in $%
Mod(\mathscr{X})$.

\begin{example}
$((\OX{\alpha })_{\alpha \in \mathcal{S}}\ ,(\rho _{\beta \alpha
}^{1})_{\alpha \subset \beta })$ is obviously an $\mathscr{X}$-module that
we denote by $\OXss{X}$.
\end{example}

\begin{example}
\label{Ex_F|U}In the context of Example \ref{ex_acop} let $\mathcal{F}\in
Mod(\OX{X})$. Then $(\mathcal{F}|U_{\alpha })_{\alpha \in \mathcal{S%
}}$ with the obvious connecting morphisms is an $\mathfrak{U}$-module that
we denote by $\mathcal{F}|\mathfrak{U}$.
\end{example}

The tensor product induces a bifunctor: 
\begin{equation*}
\otimes :Mod(\mathscr{X})\times Mod(\mathscr{X})\rightarrow Mod(\mathscr{X})
\end{equation*}%
namely if $\mathcal{F}=((\mathcal{F}_{\alpha })_{\alpha \in \mathcal{S}}\
,(\varphi _{\beta \alpha })_{\alpha \subset \beta })$, $\mathcal{G}=((%
\mathcal{G}_{\alpha })_{\alpha \in \mathcal{S}}\ ,(\psi _{\beta \alpha
})_{\alpha \subset \beta })$ $\in Mod(\mathscr{X})$ then \linebreak 
$((\mathcal{F}_{\alpha }\otimes \mathcal{G}_{\alpha })_{\alpha \in \mathcal{S}},
(\varphi _{\beta \alpha }\otimes \psi _{\beta \alpha })_{\alpha \subset \beta })$ 
is an $\mathscr{X}$-module.
\medskip

\refstepcounter{subsection} \textbf{\arabic{section}.\arabic{subsection}} 
\textbf{Alternate $\mathscr{X}$-modules. }In order to define the $F_{\sharp
} $ functor (see paragraph \ref{paragr_im_dir}) we need to construct an
alternate version for the notion of $\mathscr{X}$-module. For this, let $(I,%
\mathcal{S})$ be a simplicial complex and fix a total order on $I$. We use
the following notations:

\begin{itemize}
\item[-] if $\alpha \in \mathcal{S}$ and $j\in \left[ 0,|\alpha |\right] $,
then%
\begin{eqnarray*}
v(\alpha ;j) &=&\text{the j-th vertex of }\alpha \text{ with respect to the
order on }I\text{,\ } \\
&&\text{the counting starting from 0} \\
\sigma (\alpha ;j) &=&\alpha \setminus \{v(\alpha ;j)\}
\end{eqnarray*}

\item[-] if $\alpha \in \mathcal{S}$, $|\alpha |\geq 1$ and $j,k\in \left[
0,|\alpha |\right] $, $j\neq k$, then%
\begin{equation*}
\sigma (\alpha ;\/j,k)=\alpha \setminus \{v(\alpha ;j),v(\alpha ;k)\}
\end{equation*}
\end{itemize}

Thus, if for instance $j<k$ then $\sigma (\alpha ;j,k)=\sigma (\sigma
(\alpha ;k);j)=\sigma (\sigma (\alpha ;j);k-1).$

If $\alpha =\{i_{0},...,i_{n}\}$ and $i_{0}<i_{1}<...<i_{n}$, one checks
immediately that 
\begin{eqnarray*}
|\alpha | &=&n \\
v(\alpha ;j) &=&v_{j} \\
\sigma (\alpha ;j) &=&\{i_{0},...,\hat{\imath}_{j},...,i_{n}\} \\
\sigma (\alpha ;j,k) &=&\{i_{0},...,\hat{\imath}_{j},...,\hat{\imath}%
_{k},...,i_{n}\}
\end{eqnarray*}

Let $\mathscr{X}=((X_{\alpha })_{\alpha \in \mathcal{S}},(\rho _{\alpha
\beta })_{\alpha \subset \beta })$ be a s.s.analytic space. We use the
subscript $(\alpha ;j)$ to refer to the mappings along the edge $[\alpha
,\sigma (\alpha ;j)]$ of the simplicial complex $(I,\mathcal{S})$. Thus we
shall write $\rho _{(\alpha ;j)}$ instead of $\rho _{\sigma (\alpha
;j)\alpha }:X_{\alpha }\rightarrow X_{\sigma (\alpha ;j)}$. Similarly, if $%
\mathcal{F}=((\mathcal{F}_{\alpha })_{\alpha \in \mathcal{S}}\ ,(\varphi
_{\beta \alpha })_{\alpha \subset \beta })$ is an $\mathscr{X}$-module we
write $\varphi _{(\alpha ;j)}$ instead of $\varphi _{\alpha \sigma (\alpha
;j)}:\mathcal{F}_{\sigma (\alpha ;j)}\rightarrow \rho _{(\alpha ;j)\ast }(%
\mathcal{F}_{\alpha })$.

\begin{remark}
The family of commuting morphisms $(\rho _{\alpha \beta })_{\alpha \subset
\beta }$ can be "reconstructed" (by finite compositions) from the subfamily $%
(\rho _{(\alpha ;j)})_{(\alpha ;j)}$. More precisely, If $\alpha \in 
\mathcal{S}$, $|\alpha |\geq 1$, $j,k\in \left[ 0,|\alpha |\right] $, and,
for instance, $j<k$, then the following rectangular diagram commutes:%
\begin{equation}
\begin{CD} X_{\alpha} @>\rho_{(\alpha;j)}>> X_{\sigma(\alpha;j)}\\
@VV\rho_{(\alpha;k)}V @VV\rho_{(\sigma(\alpha;j),k-1)}V\\
X_{\sigma(\alpha;k)} @>\rho_{(\sigma(\alpha;k),j)}>> X_{\sigma(\alpha;j,k)}
\end{CD}  \tag{$D(\alpha ;j,k)$}  \label{diagr_DX}
\end{equation}

Conversely, any family of morphisms $(\rho _{(\alpha ;j)})_{(\alpha ;j)}$
such that the diagrams $D(\alpha ;j,k)$ commute, generates a family of
connecting morphisms for the family of analytic spaces $(X_{\alpha
})_{\alpha \in \mathcal{S}}$.

Similarly, the connecting morphisms of the $\mathscr{X}$-module $\mathcal{F}$%
\ are uniquely determined by the subfamily $(\varphi _{(\alpha
;j)})_{(\alpha ;j)}$ and the obvious rectangular diagrams commute:%
\begin{equation}
\begin{CD} \rho_{(\alpha;j,k)\ast}\mathcal{F}_{\alpha}
@<\rho_{(\sigma(\alpha;j),k-1)\ast}(\varphi_{(\alpha;j)})<<
\rho_{(\sigma(\alpha;j),k-1)\ast}(\mathcal{F}_{\sigma(\alpha;j)})\\
@AA\rho_{(\sigma(\alpha;k),j)\ast}({\varphi_{(\alpha;k)}})A
@AA{\varphi_{(\sigma(\alpha;j),k-1)}}A\\
\rho_{(\sigma(\alpha;k),j)\ast}(\mathcal{F}_{\sigma(\alpha;k)})
@<\varphi_{(\sigma(\alpha;k),j)}<< \mathcal{F}_{\sigma(\alpha;j,k)} \end{CD}
\tag{$D(\mathcal{F};\alpha ;j,k)$}
\end{equation}
\end{remark}

\begin{definition}

\begin{enumerate}
\item An alternate $\mathscr{X}$-module consists of a family $(\mathcal{F}%
_{\alpha })_{\alpha \in \mathcal{S}}$, where each $\mathcal{F}_{\alpha }$ is
an $\OX{\alpha }$-module, together with the family of connecting
morphisms $(\varphi _{(\alpha ;j)})_{(\alpha ;j)}$,\ $\varphi _{(\alpha ;j)}:%
\mathcal{F}_{(\alpha ;j)}\rightarrow \rho _{(\alpha ;j)\ast }(\mathcal{F}%
_{\alpha })$ such that the diagrams $D(\mathcal{F};\alpha ;j,k)$
anti-commute.

\item \label{pct_morf_alt_mod}Let $\mathcal{F}=((\mathcal{F}_{\alpha
})_{\alpha \in \mathcal{S}}\ ,(\varphi _{(\alpha ;j)})_{(\alpha ;j)})$, $%
\mathcal{G}=((\mathcal{G}_{\alpha })_{\alpha \in \mathcal{S}}\ ,(\psi
_{(\alpha ;j)})_{(\alpha ;j)})$ be alternate $\mathscr{X}$-modules. A
morphism of\ alternate\ $\mathscr{X}$-modules \ $u:$\ $\mathcal{F}%
\rightarrow \mathcal{G}$ consists of a family $(u_{\alpha })_{\alpha \in 
\mathcal{S}}$, $u_{\alpha }:\mathcal{F}_{\alpha }\rightarrow \mathcal{G}%
_{\alpha }$ morphism of $\OX{\alpha }$-modules, such that for each
pair $(\alpha ;j)$ the diagram commutes: 
\begin{equation}
\begin{CD} \mathcal{F}_{(\alpha ;j)} @>u_{(\alpha;j)}>>
\mathcal{G}_{(\alpha;j)}\\ @VV\varphi _{(\alpha ;j)}V @VV\psi _{(\alpha
;j)}V\\ \rho_{(\alpha;j)\ast }\mathcal{F}_{\alpha } @>u_{\alpha }>>
\rho_{(\alpha ;j)\ast}\mathcal{G}_{\alpha }\\ \end{CD} 
\tag{$D(\mathcal{F},\mathcal{G};\alpha ;j)$}
\end{equation}%
One denotes by $aMod(\mathscr{X})$\ the category of alternate $\mathscr{X}$%
-modules

\item With the notations at point \ref{pct_morf_alt_mod}, an anti-morphism
of alternate $\mathscr{X}$-modules is a family of morphisms $u=(u_{\alpha
})_{\alpha \in \mathcal{S}}$ such that the diagrams $D(\mathcal{F},%
\mathcal{G};\alpha ;j)$ anti-commute. A complex of alternate $\mathscr{X}$%
-modules with anti-morphism differentials will be called an alternate
complex of alternate $\mathscr{X}$-modules. One denotes by $aC(\mathscr{X})$
the category of alternate complexes of alternate $\mathscr{X}$-modules
\end{enumerate}
\end{definition}

To the edge $(\alpha ;j)$ of the simplicial complex $(I,\mathcal{S})$ we
associate the alternating coeficient $\varepsilon (\alpha ;j)=(-1)^{j}$.
Note that if $\mathcal{F}=((\mathcal{F}_{\alpha })_{\alpha \in \mathcal{S}%
},(\varphi _{(\alpha ;j)})_{(\alpha ;j)})$ is an $\mathscr{X}$-module then $%
alt(\mathcal{F})=((\mathcal{F}_{\alpha })_{\alpha \in \mathcal{S}%
},(\varepsilon (\alpha ;j)\varphi _{(\alpha ;j)})_{(\alpha ;j)})$ is an
alternate $\mathscr{X}$-module. One checks easily that $\ alt:Mod(\mathscr{X}%
)\rightarrow aMod(\mathscr{X})$ is an isomorphism of categories with an
obvious inverse that we denote by $alt^{-1}$. The functor $alt$ extends to
an isomorphism of categories $C(\mathscr{X})\rightarrow aC(\mathscr{X})$.
Indeed, if $\mathcal{F}^{\bullet }\in C(\mathscr{X}),$ $\mathcal{F}^{\bullet
}=((\mathcal{F}_{\alpha }^{\bullet })_{\alpha \in \mathcal{S}},(\varphi
_{(\alpha ;j)}^{\bullet })_{(\alpha ;j)})$ then the terms of $alt(\mathcal{F}%
^{\bullet })$ are obtained from the terms of $\mathcal{F}^{\bullet }$ via
the functor $alt$, while the differentials of each complex $\mathcal{F}%
_{\alpha }^{\bullet }$ are multiplied by $(-1)^{|\alpha |}$.

\begin{remark}
The notions of alternate $\mathscr{X}$-module and alternate complex of $%
\mathscr{X}$-modules do not depend on the total order on $I$. The $alt$
functors do. However for two total orders on $I$ there is a (non-unique)
functorial isomorphism between the two corresponding $alt$ functors.
\end{remark}

\refstepcounter{subsection}\textbf{\arabic{section}.\arabic{subsection}} %
\label{paragr_im_inv} \textbf{Inverse images. }Consider the following
setting:

\begin{itemize}
\item[-] $f:(I,\mathcal{S})\rightarrow $ $(J,\mathcal{T})$ a morphism of
simplicial complexes

\item[-] fixed total orders on $I$ and $J$.such that $f:I\rightarrow $ $J$
is increasing

\item[-] $\mathscr{X}=((X_{\alpha })_{\alpha \in \mathcal{S}}$, $(\rho
_{\alpha \beta })_{\alpha \subset \beta })$, $\mathscr{Y}=((Y_{\gamma
})_{\gamma \in \mathcal{T}}$, $(\rho _{\gamma \delta }^{\prime })_{\gamma
\subset \delta })$ s.s.analytic spaces indexed by $(I,\mathcal{S})$,
respectively $(J,\mathcal{T})$

\item[-] $F:\mathscr{X}\rightarrow \mathscr{Y}$ a morphism of s.s.analytic
spaces over $f$ (see Definition \ref{Def_morph_over_f}), that is $%
F=(F_{\alpha },F_{\alpha }^{\ast })_{\alpha \in \mathcal{S}}$ with $%
F_{\alpha }:X_{\alpha }\rightarrow Y_{f(\alpha )}$ morphism of analytic
spaces such that for $\alpha \subseteq \beta $ the following diagram
commutes: 
\begin{equation}
\begin{CD} X_{\beta} @>F_{\beta}>> Y_{f(\beta )}\\ @VV\rho_{\alpha \beta}V
@VV\rho_{f(\alpha) f(\beta)}^{\prime }V\\ X_{\alpha} @>F_{\alpha}>>
Y_{f(\alpha )} \end{CD}  \label{Diagr_XY_morph}
\end{equation}
\end{itemize}

Note that $\mathcal{S}$ is the disjoint union of the sets $(\mathcal{S}%
_{\gamma })_{\gamma \in \mathcal{T}}$, with 
\begin{equation}
\mathcal{S}_{\gamma }=\{\alpha \in \mathcal{S}|f(\alpha )=\gamma \}
\end{equation}%
Moreover, each $\mathcal{S}_{\gamma }$ is the union of the sets $(I(\gamma
,i))_{i\geq 0}$ where 
\begin{equation}
I(\gamma ,i)=\{\alpha \in \mathcal{S}|f(\alpha )=\gamma ,|\alpha |=|\gamma
|+i\}
\end{equation}%
We also set: 
\begin{equation}
I(i)=\bigcup\limits_{\gamma \in \mathcal{T}}I(\gamma ,i)
\end{equation}%
\ In particular $I(\gamma ,0)$ consists of all the simplexes of $\mathcal{S}$
which are in a one-to-one correspondence with $\gamma $ via $f$.

Let $\mathcal{G}\in Mod(\mathscr{Y})$ with $\mathcal{G}=((\mathcal{G}%
_{\gamma })_{\gamma \in \mathcal{T}},(\psi _{\delta \gamma })_{\gamma
\subset \delta })$. The inverse image $F^{\ast }(\mathcal{G})$ of $\mathcal{G%
}$ is, by definition, the $\mathscr{X}$-module with the components: 
\begin{equation}
F^{\ast }(\mathcal{G})_{\alpha }=F_{\alpha }^{\ast }(\mathcal{G}_{f(\alpha
)})\text{ for }\alpha \in \mathcal{S}
\end{equation}%
and connecting morphisms for all $\alpha \subset \beta $ 
\begin{equation}
\mu _{\beta \alpha }=F_{\beta }^{\ast }(\widetilde{\psi }_{f(\beta )f(\alpha
)})\text{ }
\end{equation}%
where%
\begin{equation}
\widetilde{\psi }_{f(\beta )f(\alpha )}:\rho _{f(\beta )f(\alpha
)}^{^{\prime }\ast }(\mathcal{G}_{f(\alpha )})\rightarrow \mathcal{G}%
_{f(\beta )}
\end{equation}%
is the morphism which corresponds via the usual adjunction isomorphism to
the connecting morphism 
\begin{equation}
\psi _{f(\beta )f(\alpha )}:\mathcal{G}_{f(\alpha )}\rightarrow \rho
_{f(\beta )f(\alpha )\ast }^{^{\prime }}(\mathcal{G}_{f(\beta )})
\end{equation}

One checks easily that the family of morphisms $(\mu _{\beta \alpha
})_{\alpha \subset \beta }$\ satisfies the required conditions.

\begin{remark}
\label{Rem_im_inv_incluz}Let $X$ be an analytic space, $\mathcal{U}%
=(U_{i})_{i\in I}$ an open covering of $X$, $\mathfrak{U}$ the s.s.analytic
space determined by $\mathcal{U}$ (see Example \ref{ex_acop}) and 
$i:$ $\mathfrak{U\rightarrow }X$ the morphism of s.s.analytic spaces given by the natural inclusions 
(see Example \ref{Ex_incluz_acop}). If $\mathcal{F}\in Mod(\OX{X})$ then the 
$\mathfrak{U}$-module $\mathcal{F}|\mathfrak{U}$ (see Example \ref{Ex_F|U}) coincides 
with $i^{\ast }(\mathcal{F})$.
\end{remark}

Let $\mathcal{G}\in Mod(\mathscr{Y})$ as above, $\mathcal{F}\in Mod(%
\mathscr{X})$ with $((\mathcal{F}_{\alpha })_{\alpha \in \mathcal{S}%
},(\varphi _{\alpha \beta })_{\alpha \subset \beta })$, and $v:F^{\ast }(%
\mathcal{G}\rightarrow \mathcal{F}$ a morphism of $\mathscr{X}$-modules. One remarks
that $v$ is completely determined by the family of morphisms $(v_{\alpha
})_{\alpha }$%
\begin{equation*}
v_{\alpha }:F_{\alpha }^{\ast }(\mathcal{G}_{f(\alpha )})\rightarrow 
\mathcal{F}_{\alpha }
\end{equation*}%
where $\alpha \in \mathcal{S}$ is such that $f|\alpha $\ is injective. More
precisely one checks directly the following lemma:

\begin{lemma}
\label{Lemma_crit_morf_im_inv}Let $\mathcal{F}\in Mod(\mathscr{X})$, $%
\mathcal{G}\in Mod(\mathscr{Y})$ as above. Then the morphisms $(v_{\alpha
})_{\alpha \in I(0)}$ determine a morphism of $\mathscr{X}$-modules $%
v:F^{\ast }(\mathcal{G})\rightarrow \mathcal{F}$ iff they verify the following
conditions:

\begin{enumerate}
\item For $\gamma \in \mathcal{T}$, $\beta \in I(\gamma ,1)$ let $\alpha
_{1} $, $\alpha _{2}\in $ $I(\gamma ,0)$ be the only two simplexes s.t. \linebreak
$\alpha _{1}$, $\alpha _{2}\subset \beta $. Then the following diagram
commutes: 
\begin{equation}
\begin{CD} F_{\beta }^{\ast }(\mathcal{G}_{\gamma}) @ >
\rho_{\alpha_{1}\beta}^{\ast }(v_{\alpha_{1}})>> \rho_{\alpha_{1}
\beta}^{\ast}(\mathcal{F}_{\alpha_{1}})\\ @VV\rho_{\alpha_{2}
\beta}^{\ast}(v_{\alpha_{2}})V @VV\widetilde{\varphi} _{\beta \alpha_{1}}V\\
\rho_{\alpha_{2} \beta}^{\ast }(\mathcal{F}_{\alpha_{2}})
@>\widetilde{\varphi} _{\beta \alpha_{2}}>> \mathcal{F}_{\beta}\\ \end{CD}
\label{Im_inv_cond_1}
\end{equation}

\item For $\gamma ,\delta \in \mathcal{T}$, $\gamma \subset \delta $, and $%
\alpha \in I(\gamma ,0)$, $\beta \in I(\delta ,0)$ with $\alpha \subset
\beta $ the following diagram commutes:%
\begin{equation}
\begin{CD} \rho_{\alpha \beta}^{\ast }F_{\alpha }^{\ast
}(\mathcal{G}_{\gamma}) @>\rho_{\alpha \beta}^{\ast }(v_{\alpha})>>
\rho_{\alpha \beta}^{\ast }(\mathcal{F}_{\alpha})\\ @VVF_{\beta }^{\ast
}(\widetilde{\psi} _{\delta \gamma})V @VV\widetilde{\varphi}_{\beta
\alpha}V\\ F_{\beta }^{\ast }(\mathcal{G}_{\delta }) @>v_{\beta }>>
\mathcal{F}_{\beta }\\ \end{CD}  \label{Im_inv_cond_2}
\end{equation}
\end{enumerate}
\end{lemma}

\begin{remark}
In the same way as above (i.e. componentwise) one can construct an inverse
image functor $F^{-1}$ for s.s.sheaves of abelian groups.
\end{remark}

\refstepcounter{subsection} \textbf{\arabic{section}.\arabic{subsection}} %
\label{paragr_im_dir} \textbf{The }$F_{\sharp }$ \textbf{functor. }We use
the setting described at the beginning of paragraph \ref{paragr_im_inv}.

Let $\mathcal{F}\in aMod(\mathscr{X})$, $\mathcal{F}=((\mathcal{F}_{\alpha
})_{\alpha \in \mathcal{S}}\ ,(\varphi _{(\alpha ;j)})_{(\alpha ;j)})$. For $%
\gamma \in \mathcal{T}$ \ 
\begin{equation}
((F_{\alpha \ast }(\mathcal{F}_{\alpha }))_{\alpha },(F_{\sigma (\alpha
;j)\ast }(\varphi _{(\alpha ,j)}))_{(\alpha ;j)})_{f(\alpha )=f(\sigma
(\alpha ;j))=\gamma }
\end{equation}%
is a multicomplex of $Y_{\gamma }$-modules (recall that multicomplex means
anti-commuting rectangles) and consider the following simple complex associated to 
this multicomplex:
\begin{equation}
...\rightarrow \prod\limits_{\alpha \in I(\gamma ,i)}F_{\alpha \ast }(%
\mathcal{F}_{\alpha })\rightarrow \prod\limits_{\alpha \in I(\gamma
,i+1)}F_{\alpha \ast }(\mathcal{F}_{\alpha })\rightarrow ... 
\tag{$C^{\bullet }(\gamma )$}
\end{equation}%
with the product indexed by $I(\gamma ,0)$ in degree $0.$ The connecting
morphisms of $\mathcal{F}$\ induce anti-morphisms $C^{\bullet }(\sigma
(\gamma ;j))\rightarrow \rho _{(\gamma ;j)\ast }(C^{\bullet }(\gamma ))$ and
one checks that $(C^{\bullet }(\gamma ))_{\gamma \in \mathcal{T}}$ is an
alternated complex of alternated $\mathscr{Y}$-modules.

If we start with an alternated complex of alternated $\mathscr{X}$-modules\ $%
\mathcal{F}$\ instead of an alternated $\mathscr{X}$-module, then $%
C^{\bullet }(\gamma )$ is a double complex where the product for $\mathcal{F}%
^{0}$ indexed by $I(\gamma ,0)$ is considered in bidegree $(0,0)$.

\begin{definition}
\begin{enumerate}
\item If $\mathcal{F}\in aMod(\mathscr{X})$ then $F_{\sharp }(\mathcal{F})$
is the alternated complex of alternated $\mathscr{Y}$-modules with 
\begin{equation}
F_{\sharp }(\mathcal{F})_{\gamma }=C^{\bullet }(\gamma )
\end{equation}
and with connecting morphisms induced by those of $\mathcal{F}$

\item If $\mathcal{F}\in aC(\mathscr{X})$ then $F_{\sharp }(\mathcal{F})$ is
the alternated complex of alternated $\mathscr{Y}$-modules where $F_{\sharp
}(\mathcal{F})_{\gamma }$ is the simple complex associated to the double
complex $C^{\bullet }(\gamma )$ and the connecting morphisms are induced by
those of $\mathcal{F}$

\item If $\mathcal{F}\in Mod(\mathscr{X})$ (respectively $\mathcal{F}\in C(%
\mathscr{X})$) then $F_{\sharp }(\mathcal{F})=alt^{-1}(F_{\sharp }(alt(%
\mathcal{F})))$
\end{enumerate}
\end{definition}

One checks easily that the definition of $F_{\sharp }$ is compatible with
the natural inclusion functors $aMod(\mathscr{X})\rightarrow aC(\mathscr{X})$
and $Mod(\mathscr{X})\rightarrow C(\mathscr{X}).$

\begin{example}
The components $(F_{\alpha }^{\ast })_{\alpha \in \mathcal{S}}$ of the
morphism $F:\mathscr{X}\rightarrow \mathscr{Y}$ determine a morphism of $%
\mathscr{Y}$-modules \ $F^{\sharp}:\OXss{Y}\rightarrow F_{\sharp }(\OXss{X})$
\end{example}

\begin{example}
If \ $f:(I,\mathcal{S})\rightarrow $ $(J,\mathcal{T})$ is bijective (in
particular if $f$ is the identity of $(I,\mathcal{S})$) then $F_{\sharp }(%
\mathcal{F})_{\gamma }=F_{\alpha \ast }(\mathcal{F}_{\alpha })$ where $%
f(\alpha )=\gamma $. Remark that if $F:X\rightarrow Y$ is a morphism of
analytic spaces and $\mathcal{F}\in Mod(\OX{X})$ then the usual
direct image $F_{\ast }(\mathcal{F})$ coincides with $F_{\sharp }\mathcal{F}$
as module over $X$\ seen as s.s.analytic space indexed by $K(pt)$ (see
Example \ref{anal_sp_Kpt})
\end{example}

\begin{example}
\label{Ex_Cech_complex}Let $X$ be an analytic space, $\mathcal{U}%
=(U_{i})_{i\in I}$ an open covering of $X$, and \linebreak 
$\mathcal{F}\in Mod(\OX{X})$. If $i:\mathfrak{U}\rightarrow X$ is the 
morphism of s.s.analytic spaces over \linebreak 
$a_{\mathcal{N(U)}}:\mathcal{N(U)}\rightarrow K(pt)$ given by
the inclusions (see Example \ref{Ex_incluz_acop}) then $i_{\sharp }(\mathcal{%
F}|\mathfrak{U})$ is the \v{C}ech complex of $\mathcal{F}$ with respect to
the covering $\mathcal{U}$. Note that the natural morphism \linebreak 
$\mathcal{F}\rightarrow i_{\sharp }(\mathcal{F}|\mathfrak{U})$ is a
quasi-isomorphism (see e.g. \cite{FAC}, \textit{chap 1, \S 4, Lemme 1}).
\end{example}

Since the Cartesian product is associative, the form of the terms of the
complex ( $C^{\bullet }(\gamma )$) implies that the functor $F_{\sharp }$\
of s.s.modules commutes with the composition of morphisms of s.s.analytic
spaces. Thus the following lemma holds:

\begin{lemma}
\label{Lemma_Comp_im_dir}Let $f:(I_{1},\mathcal{S}_{1})\rightarrow $ $(I_{2},%
\mathcal{S}_{2})$, $g:(I_{2},\mathcal{S}_{2})\rightarrow (I_{3},\mathcal{S}%
_{3})$ be morphisms of simplicial complexes and assume that we have fixed
total orders on $I_{1},$ $I_{2},$ $I_{3}$. Let $F:\mathscr{X}\rightarrow %
\mathscr{Y}$,\linebreak 
$G:\mathscr{Y}\rightarrow \mathscr{Z}$ be morphisms of s.s.analytic spaces over $f$, 
respectively $g$, where $\mathscr{X}$, respectively $\mathscr{Y}$, repectively 
$\mathscr{Z}$ is a s.s.analytic space relative to $(I_{1},\mathcal{S}_{1})$, 
Respectively to $(I_{2},\mathcal{S}_{2})$, and $(I_{3},\mathcal{S}_{3})$ 
Then if $\mathcal{F}\in aMod(\mathscr{X})$ or $\mathcal{F}\in aC(\mathscr{X})$ 
or $\mathcal{F} \in Mod(\mathscr{X})$ or $\mathcal{F}\in C(\mathscr{X})$ or 
$\mathcal{F}\in aC(\mathscr{X})$ one has 
\begin{equation}
(G\circ F)_{\sharp }(\mathcal{F})=G_{\sharp }F_{\sharp }(\mathcal{F})
\end{equation}%
%.$\blacksquare $
\end{lemma}
\begin{flushright}
$\square$
\end{flushright}
 
For $\mathcal{F}\in Mod(\mathscr{X})$ set 
\begin{equation}
F_{\ast }(\mathcal{F})=Z^{0}(F_{\sharp }(\mathcal{F}))
\end{equation}%
where $Z^{0}$ denotes the sheaf of $0$ degree cocycles. One checks that the
definition of $F_{\ast }(\mathcal{F})$ agrees with the one given in \cite%
{Flenner} \S 2.A.

There is an obvious natural inclusion morphism:%
\begin{equation}
F_{\ast }(\mathcal{F})\hookrightarrow F_{\sharp }(\mathcal{F})
\label{Morph_F*F_sharp}
\end{equation}

\begin{remark}
One checks that the morphism (\ref{Morph_F*F_sharp}) induces an isomorphism
between the respective derived functors.
\end{remark}

\begin{remark}
Lemma \ref{Lemma_Comp_im_dir} immediately implies that ($G\circ F)_{\ast }(%
\mathcal{F})=G_{\ast }F_{\ast }(\mathcal{F})$
\end{remark}

Let $\mathcal{F}\in Mod(\mathscr{X})$ be as above, $\mathcal{G}\in Mod(%
\mathscr{Y})$ with $\mathcal{G}=((\mathcal{G}_{\gamma })_{\gamma \in 
\mathcal{T}},(\psi _{\delta \gamma })_{\gamma \subset \delta })$ and let $u:%
\mathcal{G}\rightarrow F_{\sharp }(\mathcal{F})$ be a morphism of $%
\mathscr{Y}$-modules. Note that $u$ factors through $F_{\ast }(\mathcal{F})$%
. Hence $u$ is completely determined by the family of morphisms $(u_{\gamma
\alpha })_{\gamma \alpha }$ where $\gamma \in \mathcal{T}$, $\alpha \in
I(\gamma ,0)$, $\ $with%
\begin{equation}
u_{\gamma \alpha }:\mathcal{G}_{\gamma }\rightarrow F_{\alpha \ast }(%
\mathcal{F}_{\alpha })
\end{equation}%
Conversely, one checks directly the following lemma that describes the
families\ $(u_{\gamma \alpha })_{\gamma \alpha }$ as above that give a
morphism of $\mathscr{Y}$-modules $u:\mathcal{G}\rightarrow F_{\sharp }(%
\mathcal{F})$:

\begin{lemma}
\label{Lemma_crit_morf_im_dir}The morphisms $(u_{\gamma \alpha })_{\gamma
\alpha }$ are the components of a morphism of $\mathscr{Y}$-modules $u:%
\mathcal{G}\rightarrow F_{\sharp }(\mathcal{F})$ iff they verify the
following conditions:

\begin{enumerate}
\item For $\gamma \in \mathcal{T}$, $\beta \in I(\gamma ,1)$ let $\alpha
_{1} $, $\alpha _{2}\in $ $I(\gamma ,0)$ be the only two simplexes s.t. $%
\alpha _{1}$, $\alpha _{2}\subset \beta $. Then the following diagram
commutes: 
\begin{equation}
\begin{CD} \mathcal{G}_{\gamma} @>u_{\gamma \alpha_{1}}>> F_{\alpha_{1}
\ast}(\mathcal{F}_{\alpha_{1}})\\ @VVu_{\gamma \alpha_{2}}V @VVF_{\alpha_{1}
\ast}(\varphi _{\beta \alpha_{1}})V\\ F_{\alpha_{2}
\ast}(\mathcal{F}_{\alpha_{2}}) @>F_{\alpha_{2} \ast}(\varphi _{\beta
\alpha_{2}})>> F_{\beta \ast}(\mathcal{F}_{\beta})\\ \end{CD}
\label{Im_dir_cond_1}
\end{equation}

\item For $\gamma ,\delta \in \mathcal{T}$, $\gamma \subset \delta $, and $%
\alpha \in I(\gamma ,0)$, $\beta \in I(\delta ,0)$ with $\alpha \subset
\beta $ the following diagram commutes:%
\begin{equation}
\begin{CD} \mathcal{G}_{\gamma} @>u_{\gamma \alpha}>> F_{\alpha
\ast}(\mathcal{F}_{\alpha})\\ @VV \psi _{\delta \gamma}V @VVF_{\alpha
\ast}(\varphi_{\beta \alpha})V\\ \rho_{\gamma \delta \ast}^{\prime}
(\mathcal{G}_{\delta }) @>\rho_{\gamma \delta \ast}^{\prime}(u_{\delta \beta
})>> F_{\alpha \ast} \rho_{\alpha \beta \ast}(\mathcal{F}_{\beta })\\
\end{CD}  \label{Im_dir_cond_2}
\end{equation}%
%$\blacksquare $
\end{enumerate}
\end{lemma}
\begin{flushright}
$\square$
\end{flushright}

\begin{remark}
Conditions in Lemma \ref{Lemma_crit_morf_im_dir} and those in Lemma \ref%
{Lemma_crit_morf_im_inv} can be obtained from one another by adjunction.
Using the two lemmas one checks that the functor $F_{\ast }$ on $Mod(%
\mathscr{X})$ is indeed a right adjoint for $F^{\ast }$.
\end{remark}

\section{Embedding Atlases \label{Sect_Atlases}}

\refstepcounter{subsection}\textbf{\ \arabic{section}.%
\arabic{subsection}} \label{paragr_emb_triples}\textbf{Embedding triples.}
Let $i:$ $X\hookrightarrow D$ be a closed embedding of the analytic space $X$
in the complex manifold $D$. $(X,i,D)$ will be called an embedding triple. A
morphism of embedding triples $(f,\tilde{f}):(X_{1},i_{1},D_{1})\rightarrow
(X_{2},i_{2},D_{2})$ is a pair of morphisms of analytic spaces such that the
following diagram commutes:
\begin{equation}
\begin{CD} X_{1} @>{f}>> X_{2} \\ @VV{i_{1}}V @VV{i_{2}}V\\ D_{1}
@>{\tilde{f}}>> D_{2} \end{CD}  \label{morph_embed}
\end{equation}%
If $\tilde{f}$ is clear from the context we sometimes write $f$ instead of $%
(f,\tilde{f})$.

A complex manifold $D$ will be identified with the embedding triple 
$(D,id,D) $.

If $(X,i,D)$, $(Y,i^{\prime },D^{\prime })$ are embedding triples and $%
f:X\rightarrow Y$ is a morphism of analytic spaces, then, in general, there
does not exist $\tilde{f}:D\rightarrow D^{\prime }$ such that $(f,\tilde{f})$
is a morphism of embedding triples. However the following result can be
checked easily:

\begin{lemma}
\label{Lemma_complet_morph} \textbf{1.} Let $(X,i,D)$, $(Y,i^{\prime
},D^{\prime })$ be embedding triples and \linebreak $f:X\rightarrow Y$ a
morphism of analytic spaces. Then $\ j=(i_{,}i^{\prime }\circ
f):X\rightarrow D\times D^{\prime }$ is a closed embedding and we have
natural morphisms: 
\begin{equation}
(X,i,D)\overset{(id,p_{1})}{\longleftarrow }(X,j,D\times D^{\prime })\overset%
{(f,p_{2})}{\rightarrow }(Y,i^{\prime },D^{\prime })
\end{equation}%
where $p_{1},p_{2}$ are the projections. Moreover, assume we have another
embedding triple $(X,i_{1},D_{1})$ and morphisms 
\begin{equation}
(X,i,D)\overset{(id,q_{1})}{\longleftarrow }(X,i_{1},D_{1})\overset{(f,q_{2})%
}{\rightarrow }(Y,i^{\prime },D^{\prime })
\end{equation}%
Then $\alpha =(id,(q_{1},q_{2})):(X,i_{1},D_{1})\rightarrow $ $(X,j,D\times
D^{\prime })$ is the unique morphism s.t. the following diagram
commutes:
\begin{equation}
\begin{tikzcd}[column sep=1.5cm] 
(X,i,D) 
& (X,j,D \times D^{\prime}) 
\arrow{l}[swap]{(id, p_{1})} \arrow{r}{(f, p_{2})} 
& (Y,i^{\prime},D^{\prime}) \\ 
& (X,i_{1},D_{1})
\arrow{ul}{(id, q_{1})} 
\arrow{u}{\alpha} \arrow{ur}[swap]{(f, q_{2})} 
\end{tikzcd}  
\end{equation}

\textbf{2. }For $s=1$, $2$ let $(Y_{s},i_{s}^{\prime },D_{s}^{\prime })$ be
embedding triples and $f_{s}:X\rightarrow Y_{s}$ morphisms of analytic
spaces Then $\ j_{12}=(i_{,}i_{1}^{\prime }\circ f_{1},i_{2}^{\prime }\circ
f_{2}):X\rightarrow D\times D_{1}^{\prime }\times D_{2}^{\prime }$ is a
closed embedding and there exists unique morphisms $\alpha _{s}$\ s.t. the
following diagrams commute: %\medskip\ 
\begin{equation}
\begin{tikzcd}[column sep=1.5cm] 
(X,i,D) 
& (X,j,D \times D^{\prime}_{1} \times D^{\prime}_{2}) 
\arrow{l}[swap]{(id, p_{1})} \arrow{d}{\alpha_{s}}
\arrow{r}{(f_{s}, p^{\prime}_{s})} 
& (Y,i^{\prime}_{s},D^{\prime}_{s}) \\ 
& (X,j_{s},D \times D^{\prime}_{s})
\arrow{ul}{(id, p_{1})} 
\arrow{ur}[swap]{(f_{s}, p_{2})} 
\end{tikzcd}  
\end{equation}
where $j_{s}=(i_{,}i_{s}^{\prime }\circ f_{s})$ and $p_{1}, p_{2}, p_{s}^{%
\prime }$ are the obvious projections.
\end{lemma}

\begin{remark}
Under the hypothesis of Lemma \ref{Lemma_complet_morph}, if $D_{1}\subset 
\mathbb{C}^{n}$, $D_{2}\subset \mathbb{C}^{m}$ are Stein open sets
then one checks that there exists a Stein open subset $D_{1}^{^{\prime
}}\subset D_{1}$ and $\tilde{f}:D_{1}^{^{\prime }}\rightarrow D_{2}$ such
that the following diagram commutes: 
\begin{equation}
\begin{CD} D_{1} @<{i}<< D_{1}^{\prime}@>{\tilde{f}}>> D_{2}\\ @AA{i_{1}}A
@AA{i_{1}}A@AA{i_{2}}A\\ X @<{id}<< X @>{f}>> Y \end{CD}
\end{equation}%
i.e we have the diagram of embedding triples:%
\begin{equation}
(X,i_{1},D_{1})\overset{(id,i)}{\hookleftarrow }(X,i_{1},D_{1}^{^{\prime }})%
\overset{(f,\tilde{f})}{\rightarrow }(Y,i_{2},D_{2})
\end{equation}
\end{remark}

For brevity we shall say s.s.embedding triple instead of s.s.system of
embedding triples

\begin{remark}
\label{Rem_ss_embed_triple}A s.s.embedding triple $(X_{\alpha },k_{\alpha
},D_{\alpha })_{\alpha \in S}$ can be seen as a triple $(\mathscr{X},k,%
\mathscr{D})$ where $\mathscr{X}=(X_{\alpha })_{\alpha \in \mathcal{S}}$ is
s.s.analytic space, $\mathscr{D}=(D_{\alpha })_{\alpha \in \mathcal{S}}$ is
a s.s.complex manifold, and $k:\mathscr{X}\rightarrow \mathscr{D}$\ is a
morphism of s.s.analytic spaces such that each $k_{\alpha }:X_{\alpha
}\rightarrow {D}_{\alpha }$ is a closed embedding.
\end{remark}

\pagebreak
\refstepcounter{subsection} \textbf{\arabic{section}.\arabic{subsection}} 
\textbf{Embedding atlases.\label{paragr_Atlas}}

\begin{definition}
\label{def_atlas}

\begin{enumerate}
\item Let $X$ be an analytic space. An embedding atlas of $X$ consists of a
family of embedding triples $\mathcal{A}=(U_{i},k_{i},D_{i})_{i\in I}$ such
that the family \linebreak $cov(\mathcal{A})=(U_{i})_{i\in I}$ is an open covering of $%
X$. An embedding triple $(U_{i},k_{i},D_{i})$\ of $\mathcal{A}$ will be
called a chart. The pair $(X,\mathcal{A})$ will be called a locally embedded
analytic space or, sometimes, a local embedding of $X$.

\item Let $\mathcal{A}=(U_{i},k_{i},D_{i})_{i\in I}$ \ and $\mathcal{B}%
=(V_{j},k_{j}^{\prime },D_{j}^{\prime })_{j\in J}$ \ be embedding atlases of
the analytic space $X$, respectively $Y$. A morphism of locally embedded
analytic spaces $F:(X,\mathcal{A})\rightarrow (Y,\mathcal{B})$ consists of a
triple $(f,\tau ,(\tilde{f}_{i})_{i\in I})$ where:

\begin{enumerate}
\item[-] $f:X\rightarrow Y$ is a morphism of analytic spaces

\item[-] $\tau :I\rightarrow J$ is a refinement mapping such that $%
f(U_{i})\subset V_{\tau (i)}$ for all $i\in I$

\item[-] $\tilde{f}_{i}:$ $D_{i}\rightarrow D_{\tau (i)}^{\prime }$ is a
morphism of complex manifolds such that $(f|U_{i},\tilde{f}%
_{i}):(U_{i},k_{i},D_{i})\rightarrow (V_{\tau (i)},k_{\tau (i)}^{\prime
},D_{\tau (i)}^{\prime })$\ is a morphism of embedding triples.
\end{enumerate}
\end{enumerate}
\end{definition}

We say that $\mathcal{A}$ is locally finite if the open covering $cov(%
\mathcal{A})$ is locally finite.

In particular, an embedding triple $(X,i,D)$ can be seen as an embedding
atlas of $X$\ with one chart.

\begin{lemma}
\label{Lemma_assoc_ss_triples} \textbf{1. }Let $(X,\mathcal{A})$ be a
locally embedded analytic space with $\mathcal{A}%
=(U_{i},k_{i},D_{i})_{i\in I}$. There exists a s.s.embedding
triple $(\mathfrak{U},k,\mathfrak{D})=(U_{\alpha },k_{\alpha
},D_{\alpha })_{\alpha \in \mathcal{N(U)}}$ indexed by the simplicial
complex $(I,\mathcal{N}(cov(\mathcal{A})))$\ such that the embedding triples
corresponding to $0$-length simplexes coincide with the embedding triples of
the atlas $\mathcal{A}$.

\textbf{2.} If $f:(X,\mathcal{A})\rightarrow (Y,\mathcal{B})$ is a morphism
of embedded analytic spaces then $f$ induces a morphism%
\begin{equation*}
F:(\mathfrak{U},k,\mathfrak{D})\rightarrow (\mathfrak{V},k^{\prime },%
\mathfrak{D}^{\prime })
\end{equation*}%
between the respective associated s.s.embedding triples. Moreover the
following diagram commutes:%
\begin{equation*}
\begin{CD} \mathfrak{U} @>{F}>> \mathfrak{V} \\ @VV{b}V @VV{b^{\prime}}V\\ X
@>{f}>> Y \end{CD}
\end{equation*}
\end{lemma}

\begin{proof}
Take $\mathfrak{U}=((U_{\alpha })_{\alpha \in \mathcal{N}(cov(\mathcal{A}%
))},(i_{\alpha \beta })_{\alpha \subset \beta })$ to be the s.s.analytic
space corresponding to the open covering $(U_{i})_{i\in I}$ (see Example \ref%
{ex_acop}),
$\mathfrak{D}=((D_{\alpha })_{\alpha \in \mathcal{N}%
(cov(\mathcal{A}))},(p_{\alpha \beta })_{\alpha \subset \beta })$ the
s.s.complex manifold associated to the family $(D_{i})_{i\in I}$\ (see
Example \ref{ex_prod})\ and $k:\mathfrak{U\rightarrow D}$\ the morphism
deduced from the closed embeddings $k_{i}:U_{i}\rightarrow D_{i}$(one checks
that each $k_{\alpha }:U_{\alpha }\rightarrow D_{\alpha }$ is also a closed
embedding).
\end{proof}

The s.s.embedding triple from Lemma \ref{Lemma_assoc_ss_triples}\ will be
called the s.s.embedding triple associated to $(X,\mathcal{A})$.

\begin{remark}
One checks easily that the correspondence in Lemma \ref%
{Lemma_assoc_ss_triples} gives an equivalence between the category of
locally embedded analytic spaces and a subcategory of the category of
s.s.embedding triples.
\end{remark}

\begin{lemma}
\label{Lemma_lifting_morph}
\begin{enumerate}
\item Let $f:(X,\mathcal{A})\rightarrow (Y,\mathcal{B})$ be a morphism of
locally embedded analytic spaces, $\mathcal{F}\in Mod(\OX{X})$, $%
\mathcal{G}\in Mod(\OX{Y})$ and $u:\mathcal{G}\rightarrow f_{\ast }%
\mathcal{F}$ \ a morphism of $\OX{Y}$-modules. If $F:(\mathfrak{U}%
,k,\mathfrak{D})\rightarrow (\mathfrak{V},k^{\prime },\mathfrak{D}^{\prime
}) $ is the morphism induced between the s.s.systems of embedding triples
associated to $(X,\mathcal{A})$, $(Y,\mathcal{B})$ then $u$\ induces a
natural morphism: 
\begin{equation}
F^{\ast }(u):\mathcal{G}|\mathfrak{V} \rightarrow F_{\ast }(%
\mathcal{F}|\mathfrak{U})  \label{morf_f*_sus}
\end{equation}

\item Let $(X,\mathcal{A})\overset{f}{\rightarrow }(Y,\mathcal{B})\overset{g}%
{\rightarrow }(Z,\mathcal{C})$ be morphisms of locally embedded analytic
spaces and $h=g\circ f$. Let moreover $\mathcal{F} \in Mod(\OX{X})$, 
$\mathcal{G}\in Mod(\OX{Y})$, $\mathcal{H}\in Mod(\OX{Z})$
and morphisms $u:\mathcal{G}\rightarrow f_{\ast }\mathcal{F}$ $\OX{Y} $-linear, $v:\mathcal{H}\rightarrow g_{\ast }\mathcal{G}$, $w:\mathcal{%
H}\rightarrow h_{\ast }\mathcal{F}$ $\OX{Z}$-linear, such that $%
g_{\ast }(u)\circ v=w$ then one has the commutative diagram:%\medskip\ 
\begin{equation}
\begin{tikzcd}[column sep=0.2cm] 
H_{\ast}(\mathcal{F}|\mathfrak{U}) 
&& G_{\ast}(\mathcal{G}|\mathfrak{V})
\arrow{ll}[swap]{G_{\ast}(F^{\ast}(u))} \\ 
& \mathcal{H}|\mathfrak{W} 
\arrow{ul}{H^{\ast}(w)} \arrow{ur}[swap]{G^{\ast}(v)}
\end{tikzcd}
\label{diagr_3_sus}
\end{equation}
\end{enumerate}
\end{lemma}

\begin{proof}
\textbf{1. }Let $cov(\mathcal{A})\mathcal{=}(U_{i})_{i\in I}$, $cov(\mathcal{%
B})\mathcal{=}(V_{j})_{j\in J}$. For $\alpha \in \mathcal{N} (cov(\mathcal{A})), 
\gamma \in \mathcal{N}(cov(\mathcal{B}))$ such that $\tau
(\alpha )\subset \gamma $ let%
\begin{equation*}
f_{\gamma \alpha }:U_{\alpha }\rightarrow V_{\gamma }
\end{equation*}%
be the restriction of $f$ and 
\begin{equation*}
\ u_{\gamma \alpha }:\mathcal{G}|V_{\gamma }\rightarrow f_{\gamma \alpha
\ast }(\mathcal{F}|U_{\alpha })
\end{equation*}%
the restriction of $u$. One verifies that the family of morphisms 
$(u_{\gamma \alpha })_{\gamma \alpha }$, where 
\linebreak
$\gamma \in \mathcal{N}(cov(\mathcal{B}))$, $\alpha \in I(\gamma ,0)$, 
satisfies the hypothesis of Lemma \ref{Lemma_crit_morf_im_dir} and consequently 
they determine a morphism $F^{\ast }(u)$.

\textbf{2. }follows also from Lemma \ref{Lemma_crit_morf_im_dir} since one
verifies that travelling both ways along the edges of diagram (\ref{diagr_3_sus})
the two morphisms are determined by the same family of morphisms.
\end{proof}

\begin{definition}
\label{Def_f_compliant}Let $f:X\rightarrow Y$ be a morphism of analytic
spaces and let 
\linebreak 
$\mathcal{A}=(U_{i},k_{i},D_{i})_{i\in I}$ and 
$\mathcal{B}=(V_{j},k_{j}^{\prime },D_{j}^{\prime })_{j\in J}$ be 
embedding atlases of $X$, respectively $Y$. $\mathcal{A}$\ and $\mathcal{B}$ 
are said to be $f$-compliant if $cov(\mathcal{A})\prec f^{-1}(cov(\mathcal{B}))$, 
i.e.\ there exists a refinement mapping $\tau :I\rightarrow J$ such that
$f(U_{i})\subset V_{\tau (i)}$ for all $i\in I.$ Obviously the refinement
mapping $\tau $ need not be unique.
\end{definition}

Note that if $\mathcal{A}_{1}$\ and $\mathcal{A}_{2}$ are embedding atlases
of $X$, to say that $\mathcal{A}_{1}$\ and $\mathcal{A}_{2}$ are $id_{X}$%
-compliant simply means that $cov(\mathcal{A}_{1})\prec cov(\mathcal{A}_{2})$%
.

If $\mathcal{A}$\ and $\mathcal{B}$ are $f$-compliant then, in general,
there does not exist a morphism of locally embedded analytic spaces $F:(X,%
\mathcal{A})\rightarrow (Y,\mathcal{B})$ over $f$. However Lemma \ref%
{Lemma_complet_morph} immediately implies:

\begin{lemma}
\label{Lemma_complet_morph_atlas}

\begin{enumerate}
\item \label{Lemma_complet_morph_atlas_1} In the context of Definition \ref%
{Def_f_compliant} let 
\begin{equation*}
\mathcal{A}\times _{\tau }\mathcal{B}=(U_{i},k_{i\tau (i)},D_{i}\times
D_{\tau (i)}^{\prime })_{i\in I}
\end{equation*}%
Then $\mathcal{A}\times _{\tau }\mathcal{B}$ is an embedding atlas on $X$,
and the family of morphisms of embedding triples 
\begin{equation*}
(U_{i},k_{i},D_{i})\overset{(id,p_{1i})}{\longleftarrow }(U_{i},k_{i\tau
(i)},D_{i}\times D_{\tau (i)}^{\prime })\overset{(f,p_{2i})}{\rightarrow }%
(V_{\tau (i)},k_{\tau (i)}^{\prime },D_{\tau (i)}^{\prime })
\end{equation*}%
where $p_{1i}, p_{2i}$ are the projections, give the diagram of locally
embedded analytic spaces: 
\begin{equation*}
(X,\mathcal{A})\overset{(id,id,p_{1})}{\longleftarrow }(X,\mathcal{A}\times
_{\tau }\mathcal{B})\overset{(f,\tau ,p_{2})}{\rightarrow }(Y,\mathcal{B})
\end{equation*}%
Moreover, assume $\mathcal{A}^{\prime }=(U_{k},k_{k}^{\prime },D_{k}^{\prime
})_{k\in K}$ is another embedding atlas of $X$,\ and there exist morphisms 
\begin{equation*}
(X,\mathcal{A})\overset{(id,\upsilon ,q_{1})}{\longleftarrow }(X,\mathcal{A}%
^{\prime })\overset{(f,\tau \circ \upsilon ,q_{2})}{\rightarrow }(Y,\mathcal{%
B})
\end{equation*}%
Then there is a unique morphism $\alpha =(X,\mathcal{A}^{\prime
})\rightarrow $ $(X,\mathcal{A}\times _{\tau }\mathcal{B})$\ such that the
following diagram commutes:
\begin{equation}
\begin{tikzcd}[column sep=1.5cm] (X,\mathcal{A}) & (X,\mathcal{A}
\times_{\tau} \mathcal{B}) \arrow{l}[swap]{(id,id,p_{1})} \arrow{r}{(f,
\tau, p_{2})} & (Y,\mathcal{B}) \\ & (X,\mathcal{A}^{\prime})
\arrow{ul}{(id,\tau,q_{1})} \arrow{u}{\alpha} \arrow{ur}[swap]{(f,\tau \circ
\upsilon, q_{2})} \end{tikzcd}  \label{Diagr_A_prim}
\end{equation}

\item \label{Lemma_complet_morph_atlas_2} Let $\tau _{1},$ $\tau _{2}$:$%
I\rightarrow J$ be refinement mappings. Then 
\begin{equation*}
\mathcal{A}\times _{\tau _{1}\tau _{2}}\mathcal{B}=(U_{i},k_{i\tau
_{1}(i)\tau _{2}(i)},D_{i}\times D_{\tau _{1}(i)}^{\prime }\times D_{\tau
_{2}(i)}^{\prime })_{i\in I}
\end{equation*}%
is an embedding atlas on $X$ and, for $s=1,2$, one has natural morphisms:%
\begin{equation}
(X,\mathcal{A})\overset{(id,id,p_{1})}{\longleftarrow }(X,\mathcal{A}\times
_{\tau _{1}\tau _{2}}\mathcal{B})\overset{(f,\tau _{s},p_{s})}{\rightarrow }%
(Y,\mathcal{B})  \label{Diagr_tau12}
\end{equation}%
such that the following diagram commutes:%\medskip
\begin{equation}
\begin{tikzcd}[column sep=1.5cm] 
(X,\mathcal{A}) & (X,\mathcal{A} \times_{\tau_{1} \tau_{2}} \mathcal{B}) 
\arrow{l}[swap]{(id,id,p_{1})} \arrow{d}{\alpha_{s}} 
\arrow{r}{(f,\tau_{s} ,p_{s})} 
& (Y,\mathcal{B}) \\ 
& (X,\mathcal{A} \times_{\tau_{s}} \mathcal{B}) 
\arrow{ul}{(id,id,p_{1})} 
\arrow{ur}[swap]{(f,\tau_{s} ,p_{2})} 
\end{tikzcd}  \label{Diagr_tau_1_2}
\end{equation}
where $\alpha _{s}=(id,id,p_{1s})$
\end{enumerate}
\end{lemma}

\begin{corollary}
\label{Corol_Atlas_compunere}Let $X\overset{f}{\rightarrow }Y\overset{g}{%
\rightarrow }Z$ be morphisms of analytic spaces, $h=g\circ f$, and let $%
\mathcal{A}=(U_{i},k_{i},D_{i})_{i\in I}$, $\mathcal{B}=(V_{j},k_{j}^{\prime
},D_{j}^{\prime })_{j\in J}$, $\mathcal{C}=(W_{k},k_{k}^{\prime \prime},D_{k}^{\prime \prime})_{k\in
K}$ be embedding atlases of $X$, respectively $Y$, and $Z$. Assume that $%
\mathcal{A}$, $\mathcal{B}$ are $f$-compliant, $\mathcal{B}$, $\mathcal{C}$
are $g$-compliant and let $\tau :I\rightarrow J$, $\upsilon :J\rightarrow K$
be refinement mappings. Set 
\begin{equation}
\mathcal{A}\times _{\tau }\mathcal{B}\times _{\upsilon }\mathcal{C}%
=(U_{i},k_{i\tau (i)\upsilon (i)},D_{i}\times D_{\tau (i)}^{\prime }\times
D_{\upsilon (i)}^{\prime \prime})_{i\in I}
\end{equation}%
Then $\mathcal{A}\times _{\tau }\mathcal{B}\times _{\upsilon }\mathcal{C}$
is an embedding atlas on $X$, and the following diagram commutes:%
\begin{equation}
\begin{tikzcd}[column sep=0cm, row sep=scriptsize] && (X,\mathcal{A}
\times_{\tau} \mathcal{B} \times_{\upsilon} \mathcal{C})
\arrow{dl}[swap]{(id,id, p_{12})} \arrow{dr}{(f,\tau , p_{23})} \arrow[ddll,
bend right, "{(id,id,p_{1})}"'] \arrow[ddrr, bend left, "{(g \circ f,
\upsilon \circ \tau, p_{3})}"] && \\ & (X,\mathcal{A} \times_{\tau}
\mathcal{B}) \arrow{dl}[swap]{(id,id, p_{1})} \arrow{dr}{(f,\tau , p_{2})}
&& (Y,\mathcal{B} \times_{\upsilon} \mathcal{C})
\arrow{dl}[swap]{(id,id,p_{1})} \arrow{dr}{(g,\upsilon ,p_{2})} & \\
(X,\mathcal{A}) && (Y,\mathcal{B}) && (Z,\mathcal{C}) \end{tikzcd}
\label{diagr_morph_compoz}
\end{equation}%
where $p_{12},p_{23},p_{1},p_{2},p_{3}$ are the obvious projections.
\end{corollary}

\begin{remark}
\label{Rem_comut_morph_compoz}By Lemma \ref{Lemma_complet_morph_atlas}. \ref%
{Lemma_complet_morph_atlas_1} there is a unique morphism 
\begin{equation*}
\alpha :(X,\mathcal{A}\times _{\tau }\mathcal{B}\times _{\upsilon }\mathcal{C%
})\rightarrow (X,\mathcal{A}\times _{\upsilon }\mathcal{C})
\end{equation*}%
\ such that the following diagram commutes:%
\begin{equation}
\begin{tikzcd}[column sep=1.5cm] (X,\mathcal{A}) & (X,\mathcal{A}
\times_{\upsilon \circ \tau} \mathcal{C}) \arrow{l}[swap]{(id,id,p_{1})}
\arrow{r}{(g \circ f,\upsilon \circ \tau, p_{2})} & (Z,\mathcal{C}) \\ &
(X,\mathcal{A} \times_{\tau} \mathcal{B} \times_{\upsilon} \mathcal{C})
\arrow{ul}{(id,id,p_{1})} \arrow{u}{\alpha} \arrow{ur}[swap]{(g \circ
f,\upsilon \circ \tau, p_{3})} \end{tikzcd}  \label{Diagr_comut_morph_compoz}
\end{equation}
\end{remark}

\begin{remark}
\label{Rem_morph_atlas_gen}Let $(X,\mathcal{A})$, $(Y,\mathcal{B})$ be
locally embedded analytic spaces,\linebreak\ $\mathcal{A}%
=(U_{i},k_{i},D_{i})_{i\in I}$ and $\mathcal{B}=(V_{j},k_{j}^{\prime
},D_{j}^{\prime })_{j\in J}$, and $f:X\rightarrow Y$ a morphism of analytic
spaces. If $\mathcal{A}_{1}$ is an embedding atlas of $X$ over the open
covering $(U_{i}\cap f^{-1}(V_{j}))_{i,j}$ then $\mathcal{A}_{1}$, $\mathcal{%
A}$ are $id_{X}$-compliant and $\mathcal{A}_{1}$, $\mathcal{B}$ are $f$%
-compliant. Moreover, if $\mathcal{A}_{2}$ is another embedding atlas of $X$
s.t. $\mathcal{A}_{2}$, $\mathcal{A}$ are $id_{X}$-compliant and $\mathcal{A}%
_{2}$, $\mathcal{B}$ are $f$-compliant then $\mathcal{A}_{2}$, $\mathcal{A}%
_{1}$ are $id_{X}$-compliant.
\end{remark}

\section{Construction of the Dolbeault resolution\label{Sect_Resol}}

We shall extend successively the definition of the
Dolbeault-Grothendieck resolution from the classical case of complex
manifolds to that of embedding triples, then to s.s.of embedding triples
and, finally, to the case of a general analytic space with a fixed embedding
atlas. It is essential that at each extension the definition be compatible
with $\sharp$-direct images (i.e. there exists a commutative diagram similar to
diagram (\ref{diagr_3_A})). For reference purposes property \textbf{1.a}, for
instance, will be called \textbf{1.a-mfld} in the smooth case, \textbf{%
1.a-emb} in the embedded case, and \textbf{1.a-ss} in the semi-simplicial
case.

\refstepcounter{subsection}\textbf{\ \arabic{section}.\arabic{subsection}} 
\textbf{The smooth case. }Let $X$ be an $n$-dimensional complex manifold \
We regard $X$ as a locally embedded analytic space with one chart given by
the identity map. For $\mathcal{F}\in Mod(\OX{X})$\ denote by $\Dolbm{X}{F}$ the complex:%
\begin{equation}
0\longrightarrow \mathcal{E}_{X}^{0,0}\otimes _{\OX{X}}\mathcal{F}%
\longrightarrow ...\longrightarrow \mathcal{E}_{X}^{0,n}\otimes _{\OX{X}} \mathcal{F}\longrightarrow 0  \label{Compl_Dolb_F}
\end{equation}%
obtained by applying the functor $\bullet \otimes _{\OX{X}}\mathcal{F}$ to the resolution (\ref{Rezol_Dolb}) with the term containing $\mathcal{E%
}_{X}^{0,0}$ considered in degree $0$. This complex appeared first in the
proof of the duality theorems of Serre-Malgrange.

Note that $\Dolb{X}{\bullet }$ is a functor $Mod(\OX{X})\rightarrow C^{b}(X)$. We check that this functor satisfies the properties of the Theorem.

\textbf{1.d-mfld }is proved in Malgrange \cite{M}. \textbf{1.a-mfld} and 
\textbf{1.b-mfld} follow immediately from \textbf{1.d-mfld}, and \textbf{%
1.c-mfld} is well known. In \textbf{1.e-mfld }the morphism 
\begin{equation}
\Dolb{X}{\OX{X}}\otimes _{\OX{X}}\mathcal{F}\rightarrow \Dolbm{X}{F}  \label{ident_Dolb_X}
\end{equation}%
is obviously an isomorphism for any $\mathcal{F}$.

For \textbf{2-mfld }let $f:X\rightarrow Y$ be a morphism of complex
manifolds. Let%
\begin{equation}
f^{\ast }:\Dolb{Y}{\OX{Y}}\rightarrow f_{\ast }\Dolb{X}{\OX{X}}   \label{Dolb_X_f}
\end{equation}%
be the morphism given by the pullback of differential forms - it is a
morphism over the mapping $f^{\ast }:\OX{Y} \rightarrow f_{\ast }\OX{X}$. For $\mathcal{F}\in Mod(\OX{X})$, combining (\ref{ident_Dolb_X}) and (\ref{Dolb_X_f}), one gets a functorial morphism 
\begin{equation}
\Dolb{Y}{f_{\ast}\mathcal{F}}\rightarrow f_{\ast}\Dolbm{X}{F}
\label{Dolb_f_F}
\end{equation}%
and moreover the following diagram commutes:%
\begin{equation}
\begin{tikzcd}[column sep=-0.2cm] 
\Dolb{Y}{f_{\ast}\mathcal{F}}
\arrow{rr} && f_{\ast}\Dolbm{X}{F} \\ & f_{\ast
}\mathcal{F} \arrow{ul} \arrow{ur} \end{tikzcd}  \label{Diagr_Dolb_f_F}
\end{equation}

Let moreover $\mathcal{G}\in Mod(\OX{Y})$ and $u:\mathcal{%
G}\rightarrow f_{\ast}\mathcal{F}$ a morphism of $\OX{Y}$%
-modules Then the morphism (\ref{Dolb_f_F}) together with $u$ induces the
morphism of complexes: 
\begin{equation}
f^{\ast}(u):\Dolbm{Y}{G}\rightarrow f_{\ast}\Dolbm{X}{F}
\end{equation}%
and combining (\ref{Diagr_Dolb_f_F}) with \textbf{1.b-mfld} one checks that
the diagram (\ref{diagr_f*_u}) commutes in the smooth case.

\textbf{3-mfld }One starts from the obvious commutative diagram:%\smallskip
\begin{equation}
\begin{tikzcd}[column sep=-0.2cm] h_{\ast}\Dolb{X}{\OX{X}}
&& g_{\ast}\Dolb{Y}{\OX{Y}}
\arrow{ll}[swap]{g_{\ast}(f^{\ast})} \\ & \Dolb{Z}{\OX{Z}}
\arrow{ul}{h^{\ast}} \arrow{ur}[swap]{g^{\ast}} \end{tikzcd}
\label{Diagr_3_Dolb}
\end{equation}
Using isomorphism (\ref{ident_Dolb_X}) and the above diagram one checks the
case $\mathcal{G}=f_{\ast }\mathcal{F}$, $\mathcal{H}=h_{\ast }\mathcal{F}
$, and each of $u,v,w$ \ equals the respective identity. For the general
case one uses the functorial morphism (\ref{Dolb_f_F}).

Finally we consider the case of open and closed embeddings of manifolds.

\begin{remark}
\label{Ex_open_emb} Let $i:X\hookrightarrow Y$ be an open embedding of
complex manifolds, 
\linebreak
$\mathcal{F}\in Mod(\OX{Y})$, and denote by $u:%
\mathcal{F}\rightarrow i_{\ast }i^{-1}(\mathcal{F})$ the canonical adjunction 
morphism. Then obviously $\Dolb{X}{i^{-1}\mathcal{F}}\simeq i^{-1}\Dolbm{Y}{F}$ 
and, hence, $i^{\ast }(u)$ coincides with the adjunction morphism 
$\Dolbm{Y}{F}\rightarrow i_{\ast }i^{-1}\Dolbm{Y}{F}$.
\end{remark}

\begin{proposition}
\label{Prop_qiso_manif} Let $i:X\hookrightarrow Y$ be a closed embedding of
manifolds and $\mathcal{F}\in Mod(\OX{X})$ Then the natural morphism 
\begin{equation}
i^{\ast }(id):\Dolb{Y}{i_{\ast }\mathcal{F}}\rightarrow i_{\ast } \Dolbm{X}{F}  \label{i*_closed_emb}
\end{equation}%
is a quasiisomorphism.
\end{proposition}

\begin{proof}
$i^{\ast }(id)$ is a morphism between two soft resolutions of $i_{\ast }%
\mathcal{F}$.\medskip
\end{proof}

\refstepcounter{subsection} \textbf{\arabic{section}.\arabic{subsection}} %
\label{paragr_emb_case}\textbf{The embedded case. }Let $(X,i,D)$ be an
embedding triple. We regard $(X,i,D)$ as a locally embedded analytic space
with one chart. Let $\mathcal{F}\in Mod(\OX{X})$. We use the
notation $\Dolbm{i\,}{F}$ for the restriction to $X$ of the
complex $\Dolb{D}{i_{\ast }\mathcal{F}}$: 
\begin{equation}
0\longrightarrow \mathcal{E}_{D}^{0,0}\otimes _{\OX{D}}i_{\ast }%
\mathcal{F}\longrightarrow ...\longrightarrow \mathcal{E}_{D}^{0,n}\otimes _{\OX{D}}i_{\ast }\mathcal{F}\longrightarrow 0  \label{Compl_Dolb_i}
\end{equation}%
Remark that each term of $\Dolb{D}{i_{\ast }\mathcal{F}}$, is null
outside $X$ and, furthermore, has a structure of $i_{\ast }\OX{X}$%
-module deduced via the natural morphism $\OX{D}\rightarrow i_{\ast}
\OX{X}$. If $\mathcal{I}_{X}\subset \OX{D}$ is the ideal that defines $X$ 
as an analytic subspace of $D$ (i.e. $i_{\ast }\OX{X}\thickapprox \OX{D}/\mathcal{I}_{X}$) then it is immediate to see that $\Dolb{i\,}{\OX{X}}$ is isomorphic with the restriction to $X$ of the complex 
\begin{equation}
0\longrightarrow \mathcal{E}_{D}^{0,0}/\mathcal{I}_{X}\mathcal{E}%
_{D}^{0,0}\longrightarrow ...\longrightarrow \mathcal{E}_{D}^{0,n}/\mathcal{I%
}_{X}\mathcal{E}_{D}^{0,n}\longrightarrow 0  \label{Compl_Dolb_i1}
\end{equation}%
and, moreover, $\Dolbm{i\,}{F}$ is isomorphic with the restriction to $X$ of the complex: 
\begin{equation}
0\longrightarrow \mathcal{E}_{D}^{0,0}/\mathcal{I}_{X}\mathcal{E}%
_{D}^{0,0}\otimes _{\OX{D}}i_{\ast }\mathcal{F}\longrightarrow
...\longrightarrow \mathcal{E}_{D}^{0,n}/\mathcal{I}_{X}\mathcal{E}%
_{D}^{0,n}\otimes _{\OX{D}}i_{\ast }\mathcal{F}\longrightarrow 0
\end{equation}

Thus $\Dolb{i\,}{\bullet}$ is a functor $Mod(\OX{X})\rightarrow C^{b}(X)$ and 
\begin{equation}
\Dolbm{i\,}{F}=i^{-1}\Dolb{D}{i_{\ast }\mathcal{F}} 
\label{Def_Dolb_i}
\end{equation}%
Moreover, applying $i_{\ast }$\ one obtains an isomorphism on $D$: 
\begin{equation}
\Dolb{D}{i_{\ast }\mathcal{F}}\approx i_{\ast }\Dolbm{i\,}{F}  \label{i*_Def_Dolb_i}
\end{equation}

The functor $\Dolb{i\,}{\bullet}$ coincides with $\Dolb{X}{\bullet}$ 
if $X$ is a complex manifold (recall that we identify $X$ with the embedding triple $(X,id,X)$).

The properties of $\Dolb{i\,}{\bullet}$ are obtained from those of $\Dolb{D}{\bullet}$ via $i^{-1}$ and $i_{\ast}$
(see (\ref{Def_Dolb_i}) and (\ref{i*_Def_Dolb_i})). Indeed, since $i^{-1}$ 
is an exact functor the statements \textbf{1.a-emb}, through \textbf{1.d-emb}
for $\Dolb{i\,}{\bullet}$ follow immediately from the respective statements in the smooth case. Moreover, for any $\mathcal{F} \in Mod(\OX{X})$ the isomorphism (\ref{ident_Dolb_X}) implies that one has an isomorphism 
\begin{equation}
\Dolb{i\,}{\OX{X}}\otimes _{\OX{X}} \mathcal{F}\rightarrow 
\Dolbm{i\,}{F}  \label{ident_Dolb_i}
\end{equation}

\qquad \textbf{2-emb} Let \ $f=(f,\tilde{f}):(X,i_{1},D_{1})\rightarrow
(Y,i_{2},D_{2})$ be a morphism of embedding triples. On $D_{2}$ one has the
sequence of morphisms:%
\begin{equation}
\Dolb{D_{2}}{i_{2\ast}\OX{Y}}\rightarrow \Dolb{D_{2}}{\tilde{f}_{\ast }i_{1\ast }\OX{X}} \rightarrow  \tilde{f}_{\ast }\Dolb{D_{1}}{i_{1\ast}\OX{X}}
\end{equation}

By definition $f^{\ast }$ is the morphism obtained by restricting to $X$ the
composition of morphisms above (which comes down to applying $i_{2}^{-1}$):%
\begin{equation}
f^{\ast }:\Dolb{i_{2}}{\OX{Y}} \rightarrow f_{\ast}\Dolb{i_{1}}{\OX{X}}  \label{Dolb_f_i}
\end{equation}

Combining (\ref{ident_Dolb_i}) and (\ref{Dolb_f_i}), one gets for $\mathcal{F%
}\in Mod(\OX{X})$ a functorial morphism 
\begin{equation}
\Dolb{i_{2}}{f_{\ast }\mathcal{F}}\rightarrow f_{\ast }\Dolbm{i_{1}\,}{F}  \label{Dolb_f*_i}
\end{equation}

Let moreover $\mathcal{G}\in Mod(\OX{Y})$ and $u:\mathcal{%
G}\rightarrow f_{\ast }\mathcal{F}$ a morphism of $\OX{Y}$%
-modules Then $f^{\ast }(u)$ is by definition the composition of the above
morphism with the morphism induced by $u$: 
\begin{equation}
f^{\ast}(u):\Dolbm{i_{2}\,}{G} \rightarrow f_{\ast} \Dolbm{i_{1}\,}{F}
\end{equation}%
Finally, diagram (\ref{diagr_f*_u}) on $Y$ commutes iff its extension by $0$
to $D_{2}$ commutes, and this is true by \textbf{2-mfld}.

\begin{example}
The embedding morphism $i:X\hookrightarrow D$ can be seen as the morphism of
embedding triples $(i,id_{D})$ and $i^{\ast }$ is the natural quotient
morphism 
\begin{equation}
\Dolb{D}{\OX{D}}\rightarrow i_{\ast }\Dolb{i\,}{\OX{X}}
\end{equation}
\end{example}

For two embedings of the same analytic space one checks easily the following
result:

\begin{remark}
\label{Rem_qizo_id}Let $f=(id,\tilde{f}):(X,i_{1},D_{1})\rightarrow
(X,i_{2},D_{2})$ be a morphism of embedding triples over the same analytic
space. If $\mathcal{F}\in Mod(\OX{X})$ then one checks easily that
the natural morphism $f^{\ast}(id):\Dolbm{i_{2}\,}{F}\rightarrow \Dolbm{i_{1}\,}{F}$ is a quasi-isomorphism.
\end{remark}

\begin{remark}
In the above setting assume we only have a morphism of analytic spaces 
\linebreak
$f:X\rightarrow Y$ instead of a morphism\ of embedding triples 
$(f,\tilde{f}):(X,i_{1},D_{1})\rightarrow (Y,i_{2},D_{2})$. Then there 
may not exist a "direct" morphism
$f^{\ast}(u):\Dolbm{i_{2}\,}{G}\rightarrow f_{\ast}\Dolbm{i_{1}\,}{F}$. 
However by Lemma \ref{Lemma_complet_morph} we have the sequence of mappings:
\begin{equation}
(X,i_{1},D_{1})\overset{(id,p_{1})}{\longleftarrow }(X,i_{12},D_{1}\times
D_{2})\overset{(f,p_{2})}{\longrightarrow }(Y,i_{2},D_{2})
\end{equation}%
and consequently the sequence of morphisms on $Y$:%
\begin{equation}
f_{\ast}\Dolbm{i_{1}\,}{F}\overset{f^{\ast}(id)}{\longrightarrow }
f_{\ast}\Dolbm{i_{12}\,}{F}\overset{(f^{\ast}(u))}{\longleftarrow}
\Dolbm{i_{2}\,}{G}
\label{morph_phi*_deriv}
\end{equation}%
Since the components of $\Dolb{\bullet \,}{\mathcal{F}}$ are soft
sheaves, Remark \ref{Rem_qizo_id} implies that the first morphism above is a
quasi-isomorphism and hence one has a morphism
\linebreak 
$\Dolbm{i_{2}\,}{G}\rightarrow f_{\ast}\Dolbm{i_{1}\,}{F}$ in the derived category 
$D^{b}(\OX{Y})$, that we also denote by $f^{\ast}(u)$. Moreover, if $(X,i_{3},D_{3})$ 
is another embedding triple with mappings:
\begin{equation}
(X,i_{1},D_{1})\overset{(id,q_{1})}{\longleftarrow }(X,i_{3},D_{3})\overset{%
(f,q_{2})}{\longrightarrow }(Y,i_{2},D_{2})
\end{equation}%
then Lemma \ref{Lemma_complet_morph}\ also implies that the morphism in the
derived category $D^{b}(\OX{Y})$: 
\begin{equation}
f_{\ast}\Dolbm{i_{1}\,}{F}\overset{f^{\ast }(id)}{%
\longrightarrow }f_{\ast}\Dolbm{i_{3}\,}{F}\overset{(f^{\ast
}(u))}{\longleftarrow }\Dolbm{i_{2}\,}{G}
\end{equation}%
coincides with the morphism (\ref{morph_phi*_deriv}).
\end{remark}

To prove \textbf{3-emb }in the embedded case extend the sheaves $\mathcal{F}$%
, $\mathcal{G}$, $\mathcal{H}$ with $0$ to $D_{1}$, $D_{2}$, respectively $%
D_{3}$. It is easy to see that $i_{1\ast }\mathcal{F}$, $i_{2\ast }\mathcal{G%
}$, $i_{3\ast }\mathcal{H}$ satisfy the hypothesis of \textbf{3-mfld}. Thus
we get the commutative diagram on $D_{3}$:

\begin{equation}
\begin{tikzcd}[column sep=-0.2cm]
\tilde{h}_{\ast}\Dolb{D_{1}}{i_{1\ast}\mathcal{F}} &&
\tilde{g}_{\ast}\Dolb{D_{2}}{i_{2\ast}\mathcal{G}} \arrow{ll} \\ &
\Dolb{D_{3}}{i_{3\ast}\mathcal{H}} \arrow{ul} \arrow{ur} \end{tikzcd}
\label{diagr_3_i_sus}
\end{equation}
Restricting the above diagram to $Z$ (i.e. applying $i_{3}^{-1}$) one gets
the result.\smallskip \smallskip

\refstepcounter{subsection} \textbf{\arabic{section}.\arabic{subsection}} 
\textbf{The semi-simplicial case. }We extend the functor $\Dolbn$ to
the semi-simplicial context for technical reasons. While a s.s.analytic
space is not a particular case of analytic space with an embedding atlas,
most of the properties in Theorem \ref{Theor_Dolb} apply.

Let $\mathscr{X}=((X_{\alpha })_{\alpha \in \mathcal{S}}$, $(\rho _{\alpha
\beta })_{\alpha \subset \beta })$ be a s.s.complex manifold indexed by the
simplicial complex $(I, \mathcal{S})$, and let $\mathcal{F}\in Mod(\mathscr{X})$,  
$\mathcal{F}=((\mathcal{F}_{\alpha })_{\alpha \in \mathcal{S%
}}\ ,(\varphi _{\beta \alpha })_{\alpha \subset \beta })$. We set: 
\begin{equation}
\Dolbss{X}{F}=(\Dolb{X_{\alpha }}{\mathcal{F}_{\alpha }}_{\alpha \in \mathcal{S}}
\ ,(\rho _{\alpha \beta }^{\ast }(\varphi _{\beta \alpha }))_{\alpha \subset \beta })
\end{equation}%
By property \textbf{3-mfld} $\Dolbss{X}{F}$ is a
complex of $\mathscr{X}$-modules; furthermore, it is null in degrees $<0$

Note that if $\mathscr{X}$ is a s.s.complex manifold relative to the
simplicial complex $K(pt)$ (i.e. a complex manifold - see Example \ref%
{anal_sp_Kpt})\ the functor $\Dolb{\mathscr{X}\,}{\bullet}$
 coincides with the functor $\Dolbn$ from the smooth case.

Let now $(\mathscr{X},k,\mathscr{D})$ be a s.s.embedding triple indexed by
the simplicial complex $(I, \mathcal{S})$ (see Remark \ref{Rem_ss_embed_triple}), 
where $\mathscr{X}=(X_{\alpha })_{\alpha \in \mathcal{S}}$ is a s.s.analytic space, 
$\mathscr{D}=(D_{\alpha })_{\alpha \in \mathcal{S}}$ is a s.s.complex manifold, 
and each component $k_{\alpha}:X_{\alpha }\rightarrow D_{\alpha }$ of the morphism 
$k:$ $\mathscr{X} \rightarrow \mathscr{D}$ is a closed embedding.

If $\mathcal{F}\in Mod(\mathscr{X})$, $\mathcal{F}=(\mathcal{F}_{\alpha
})_{\alpha \in \mathcal{S}}$ we set: 
\begin{equation}
\Dolbm{k\,}{F}=k^{\ast }\Dolb{\mathscr{D}}{k_{\sharp}\mathcal{F}}
\end{equation}%
Note that in this case $k_{\sharp}=k_{\ast}$ and for each $\alpha $, $%
(k_{\sharp }\mathcal{F})_{\alpha }=k_{\alpha \ast }\mathcal{F}_{\alpha }$.
Hence, for $\alpha \in \mathcal{S}$%
\begin{equation}
\Dolb{k\,}{\mathcal{F}}_{\alpha }=\Dolb{k_{\alpha }\,}{\mathcal{F}_{\alpha }}
\end{equation}%
Thus $\Dolb{k\,}{\bullet }$ is a functor $Mod(\mathscr{X})\rightarrow C^{+}(\mathscr{X})$. 
Furthermore, $\Dolbm{k\,}{F}$ is null in degrees $<0$.

In what follows we shall check the properties of the functor $\Dolb{k\,}{\bullet }$. 
The treatment for $\Dolb{\mathscr{X}\,}{\bullet }$, where $\mathscr{X}$\ is a s.s.complex manifold, is similar.

Properties \textbf{1.a-ss}, \textbf{1.b-ss}, and \textbf{1.e-ss} follow from
the corresponding properties in the embedded case applied for each component 
$\alpha \in \mathcal{S}$. Properties \textbf{1.c} and \textbf{1.d} have no
sense in this context. However, they can be replaced by the following
statements that follow immediately from the embedded case:

\textbf{1.c'-ss} \textit{The terms of } $\Dolbm{k\,}{F}_{\alpha }=
\Dolb{X_{\alpha }}{\mathcal{F}_{\alpha }}$\textit{\ are
soft sheaves for all }$\alpha \in \mathcal{S}$.

\textbf{1.d'-ss }\textit{The terms of } $\Dolb{k\,}{\OXss{X}}_{\alpha }=
\Dolb{X_{\alpha }}{\OX{\alpha }}\ $%
\textit{are }$\OX{X_{\alpha }}$\textit{--flat for all }$\alpha \in \mathcal{S}$

\textbf{2-ss }Consider the following data:

\begin{itemize}
\item[-] $\tau :(I,\mathcal{S})\rightarrow (J,\mathcal{T})$ a mapping of
simplicial complexes

\item[-] $F:(\mathscr{X},k,\mathscr{D})\rightarrow (\mathscr{Y},k^{\prime },%
\mathscr{D}^{\prime })$ a morphism of s.s.embedding triples over $\tau $

\item[-] $\mathcal{F}\in Mod(\mathscr{X})$, $\mathcal{F}=(\mathcal{F}%
_{\alpha })_{\alpha \in \mathcal{S}}$, $\mathcal{G}\in Mod(\mathscr{Y})$, 
$\mathcal{G}=(\mathcal{G}_{\gamma })_{\gamma \in \mathcal{T}}$, and  
$u: \mathcal{G}\rightarrow F_{\ast }\mathcal{F}$ a morphism of $\mathscr{Y}$%
-modules
\end{itemize}

For $\gamma \in \mathcal{T}$ and $\alpha \in I(\gamma ,0)$ consider the
natural morphism:%
\begin{equation}
u_{\gamma \alpha }:\Dolb{k_{\gamma }^{\prime }}{\mathcal{G}_{\gamma }}
\rightarrow F_{\alpha \ast }\Dolb{k_{\alpha }}{\mathcal{F}_{\alpha }}
\end{equation}%
According to property \textbf{3-emb} one sees that the family of morphisms $%
(u_{\gamma \alpha })_{\gamma \alpha }$ satisfies the hypothesis of Lemma \ref%
{Lemma_crit_morf_im_dir}. Hence it induces a morphism%
\begin{equation}
F^{\sharp }(u):\Dolb{k^{\prime }\,}{\mathcal{G}}\rightarrow F_{\sharp }%
\Dolbm{k\,}{F}  \label{morph_pullback_sharp}
\end{equation}

\textbf{3-ss }One checks directly that on each component $\gamma \in 
\mathcal{T}$ the diagram commutes, which comes down to property \textbf{3-emb}.
\medskip

\refstepcounter{subsection} \textbf{\arabic{section}.\arabic{subsection}%
\label{Paragr_loc_emb_sp}} \textbf{The case of a locally embedded analytic
space. }Let $(X,\mathcal{A})$ be a locally embedded analytic space, where
$\mathcal{A}=(U_{i},k_{i},D_{i})_{i\in I}$. We fix the following notations
(see paragraph \ref{paragr_Atlas}):

\begin{itemize}
\item[-] $\mathcal{U}=(U_{i})_{i\in I}$ be the open covering of $X$\
corresponding to the atlas $\mathcal{A}$

\item[-] $(\mathfrak{U},k,\mathfrak{D})$ be the s.s.embedding triple
associated to $(X,\mathcal{A})$ and $b:\mathfrak{U}\rightarrow X$ the
natural morphism given by the inclusions (see Lemma \ref%
{Lemma_assoc_ss_triples})
\end{itemize}

For $\mathcal{F}\in Mod(\OX{X})$ set:%
\begin{equation}
\DolbA{A}{F}=b_{\sharp }\Dolb{k\,}{\mathcal{F}|\mathfrak{U}}  \label{Def_Dolb_A}
\end{equation}
Thus $\DolbA{A}{F}$ is the simple complex associated to the double complex:

\begin{equation}
0\rightarrow \prod\limits_{|\alpha |=0}b_{\alpha \ast }\Dolb{k_{\alpha }\,}
{\mathcal{F}|U_{\alpha }}\rightarrow \prod\limits_{|\alpha |=1} b_{\alpha 
\ast }\Dolb{k_{\alpha }\,}{\mathcal{F}|U_{\alpha}}\rightarrow ...  \label{Complex_Dolb_A}
\end{equation}%
where $\alpha \in \mathcal{N(U)}$.

\textbf{1.f }If the atlas $\mathcal{A}$ consists of only one chart $(X,i,D)$ 
(i.e. $X$ is embedded in the complex manifold $D$) then the functor $\DolbA{A}{F}$ coincides with the functor $\Dolbm{i\,}{F}$ defined at paragraph \ref{paragr_emb_case} (the embedded case). In particular, if $X$ is a complex manifold and $\mathcal{A}$ consists of the chart $(X,id,X)$, then $\Dolb{\mathcal{A}\,}{\OX{X}}$ is the usual 
Dolbeault-Grothendieck resolution on $X$.

\textbf{1.a }Let%
\begin{equation*}
0\rightarrow \mathcal{F}_{1}\rightarrow \mathcal{F}_{2}\rightarrow \mathcal{F%
}_{3}\rightarrow 0
\end{equation*}%
be a short exact sequence of $\OX{X}$-modules. By properties 
\textbf{1.a-ss }(exactness of $\Dolb{k\,}{\bullet }$ ) and \textbf{1.c'-ss }it follows that 
\begin{equation*}
0\rightarrow \Dolb{k\,}{\mathcal{F}_{1}|\mathfrak{U}} \rightarrow 
\Dolb{k\,}{\mathcal{F}_{2}|\mathfrak{U}} \rightarrow \Dolb{k\,}
{\mathcal{F}_{3}|\mathfrak{U}} \rightarrow 0
\end{equation*}%
is an exact sequence of complexes of $\mathfrak{U}$-modules such that the
terms of every $\Dolb{k\,}{\mathcal{F}_{i}}_{\alpha }$ are soft sheaves. Hence the image by $i_{\alpha \ast }$ of the exact sequence on $U_{\alpha }$ is an exact sequence on $X$, and the exactness of $\Dolb{\mathcal{A}\,}{\bullet}$ follows by using (\ref{Complex_Dolb_A}).

\textbf{1.c }$b_{\alpha \ast }(\Dolb{k\,}{\mathcal{F}}_{\alpha})=
b_{\alpha \ast }(\Dolb{k_{\alpha }}{\mathcal{F}|U_{\alpha }})$
has soft terms. Thus the terms of $\DolbA{A}{F}$ are cartesian products of soft sheaves and consequently soft.

\textbf{1.b }The morphism is given by the composition:%
\begin{equation}
\mathcal{F}\overset{u}{\rightarrow }b_{\sharp }(\mathcal{F}|\mathfrak{U})%
\overset{v}{\rightarrow }b_{\sharp} \Dolb{k\,}{\mathcal{F}|\mathfrak{U}}=\DolbA{A}{F} \label{morph_augment}
\end{equation}%
As remarked in Example \ref{Ex_Cech_complex}, $b_{\sharp }(\mathcal{F}|%
\mathfrak{U})$ coincides with the \v{C}ech complex of $\mathcal{F}$ with
respect to the open covering $\mathcal{U}$ and $u$ is a quasi-isomorphism.

To see that $v$ is also a quasi-isomorphism we restrict ourselves to an open
set $U_{j}$ of the covering $\mathcal{U}$. $\DolbA{A}{F}|U_{j}$ is the simple complex associated to the double complex (see formula (\ref{Complex_Dolb_A})):%
\begin{equation}
K^{p,q}=\prod\limits_{|\alpha |=p}b_{j\alpha \ast } \Dolbn^{q}(k_{\alpha }\, ,\mathcal{F}|U_{\alpha })|U_{\alpha }\cap U_{j}
\end{equation}%
where $b_{j\alpha }:U_{\alpha }\cap U_{j}\hookrightarrow U_{j}$ is the
natural inclusion.

The terms of the second drawer of the spectral sequence associated to 
$K^{\cdot \cdot }$ are:%
\begin{equation}
E_{2}^{p,q}=\left\{ 
\begin{array}{cc}
\mathcal{F}|U_{j} & \text{if }p=q=0 \\ 
0 & \text{otherwise}%
\end{array}%
\right.  \label{K_spect_seq}
\end{equation}%
Indeed, taking the cohomology of $K^{\cdot \cdot }$ in the $q$-direction,
one obtains for $q\geq 0$, the following \v{C}ech-type complex relative to
the covering $\mathcal{U}\cap U_{j}:$%
\begin{equation}
\cdots \rightarrow \prod\limits_{|\alpha |=p}R^{q}b_{j\alpha \ast }(\mathcal{%
F}|U_{\alpha }\cap U_{j})\rightarrow \prod\limits_{|\alpha
|=p+1}R^{q}b_{j\alpha \ast }(\mathcal{F}|U_{\alpha }\cap U_{j})
\rightarrow \cdots   \tag{C(q)}  \label{Complex_R(q)}
\end{equation}%
Note that $U_{j}$ is among the open sets of the covering $\mathcal{U}\cap %
U_{j}$ and that $b_{jj}=id_{U_{j}}$. Using a homotopy argument similar to
that in \cite{FAC}, chap 1, \S 3 Proposition 3 and \S 4, Lemma 1 one checks
that the cohomology of $C(q)$:%
\begin{equation*}
H^{p}(C(q))=\left\{ 
\begin{array}{cc}
R^{q}b_{jj\ast }(\mathcal{F}|U_{j})=R^{q}id_{\ast }(\mathcal{F}|U_{j}) & 
\text{if }p=0 \\ 
0 & \text{otherwise}%
\end{array}%
\right.
\end{equation*}%
which proves the claim. Moreover, one also checks that $v$ induces an
isomorphism between the second drawers of the spectral sequences associated
to $\mathcal{C}^{\bullet }(\mathcal{U},\mathcal{F})|U_{j}$ and $K^{\cdot
\cdot }$ and, consequently, is a quasi-isomorphism.

\bigskip \textbf{1.e} and \textbf{1.d }The morphism (\ref{ident_Dolb_A}) is
induced by the natural morphisms:%
\begin{equation}
\prod\limits_{|\alpha |=p}b_{\alpha \ast }\Dolb{k_{\alpha }\,}
{\OX{X}|U_{\alpha }}\otimes _{\OX{X}}\mathcal{F}\rightarrow
\prod\limits_{|\alpha |=p}b_{\alpha \ast }\Dolb{k_{\alpha }\,}
{\mathcal{F}|U_{\alpha }}  \label{morph_comut_prod_cartez}
\end{equation}%
If $\mathcal{F}$ is a coherent $\OX{X}$-module then (\ref%
{morph_comut_prod_cartez}) is an isomorphism. Indeed if $\mathcal{F}=\OX{X}$
or $\mathcal{F}=\OX{X}^{p}$ then the statement is clear. Hence, using local
presentations of $\mathcal{F}$%
\begin{equation*}
\OX{X}^{p}\rightarrow \OX{X}^{q}\rightarrow \mathcal{F} \rightarrow 0
\end{equation*}%
it follows that (\ref{morph_comut_prod_cartez}) is a local isomorphism, and
consequently an isomorphism. Note that if $\mathcal{F}$ is not coherent then %
(\ref{morph_comut_prod_cartez}) may not be an isomorphism.

To prove the flatness of $\Dolb{\mathcal{A}\,}{\OX{X}}$ it is enough to check 
that the following sequence:%
\begin{equation}
0\rightarrow \Dolb{\mathcal{A}\,}{\OX{X}} \otimes _{\OX{X}} \mathcal{I}
\rightarrow \Dolb{\mathcal{A}\,}{\OX{X}}
\end{equation}%
is exact, where $\mathcal{I}\subset \OX{X}|U$ is any coherent ideal on some 
open set $U\subset X$. This is implied by the following commutative diagram:%
\begin{equation}
\begin{CD} @. \Dolb{\mathcal{A}\,}{\OX{X}}\otimes _{\OX{X}} \mathcal{I} @>>>
\Dolb{\mathcal{A}\,}{\OX{X}}\\ @. @VV\wr V @|\\ 0 @>>>
\Dolb{\mathcal{A}\,}{\mathcal{I}} @>>>
\Dolb{\mathcal{A}\,}{\OX{X}} 
\end{CD}
\end{equation}%
since the lower row is exact (exactness of $\Dolb{\mathcal{A}\,}{\bullet }$ ) and 
the left hand side vertical arrow is an isomorphism by property \textbf{1.e} in the coherent case.

Finally, to prove that the morphism (\ref{ident_Dolb_A}) is a
quasi-isomorphism consider the commutative diagram:
\begin{equation}
\begin{CD}
\Dolb{\mathcal{A}\,}{\OX{X}\otimes_{\OX{X}}%
\mathcal{F}} @>>> \DolbA{A}{F}\\ @AAA @AAA\\
\OX{X}\otimes_{\OX{X}}\mathcal{F} @>>> \mathcal{F} \end{CD}
\medskip
\end{equation} 
By the $\OX{X}$-flatness of $\Dolb{\mathcal{A}\,}{\OX{X}})$ and \textbf{1.c, }
the vertical arrows are quasi-isomorphisms, and the lower horizontal arrow is an isomorphism
which yields the result.

\textbf{2 } Let $(\mathfrak{U},k,\mathfrak{D})$, $(\mathfrak{V},k^{\prime },%
\mathfrak{D}^{\prime })$ be the s.s.embedding triples associated to $(X,%
\mathcal{A})$, $(Y,\mathcal{B})$ and $b:\mathfrak{U}\rightarrow X$, 
$b^{\prime }:\mathfrak{V}\rightarrow Y$ the morphisms given by the
inclusions. Let%
\begin{equation}
F:(\mathfrak{U},k,\mathfrak{D})\rightarrow (\mathfrak{V},k^{\prime },%
\mathfrak{D}^{\prime })
\end{equation}%
be the morphism induced by $f:(X,\mathcal{A})\rightarrow (Y,\mathcal{B})$
(see Remark \ref{Lemma_lifting_morph} \textbf{2}) and 
\begin{equation}
F^{\ast }(u):\mathcal{G}|\mathfrak{V} \rightarrow F_{\ast }
(\mathcal{F}|\mathfrak{U})
\end{equation}%
be the morphism induced by $u:\mathcal{G}\rightarrow f_{\ast }\mathcal{F}$
(see Lemma \ref{Lemma_lifting_morph}. \textbf{1}). By property \textbf{2-ss}
there exists a natural morphism:%
\begin{equation}
F^{\sharp }(F^{\ast }(u)):\Dolb{k^{\prime}\,}{\mathcal{G}|{\mathfrak{V}}} \rightarrow F_{\sharp }\Dolb{k\,}{\mathcal{F}|\mathfrak{U}}
\end{equation}%
Now apply $b_{\sharp }^{\prime }$ and use Lemma \ref{Lemma_Comp_im_dir} to
obtain the morphism:%
\begin{equation}
f^{\ast }(u):\DolbA{B}{G} \rightarrow f_{\ast }\DolbA{A}{F}
\end{equation}

\begin{remark}
\label{Rem_q_izo_over_id}Assume $Y=X$, $f=(id,\tau ,(\tilde{f}_{i})_{i\in
I}):(X,\mathcal{A})\rightarrow (X,\mathcal{B})$, $\mathcal{G}=\mathcal{F}$, 
and $u=id$. Then the morphism defined above%
\begin{equation}
f^{\ast }(id):\DolbA{B}{F} \rightarrow \DolbA{A}{F}
\end{equation}%
is a quasi-isomorphism, since it is a morphism between two resolutions of 
$\mathcal{F}$.
\end{remark}

\textbf{3 } Let $(\mathfrak{U},k,\mathfrak{D}),$ $(\mathfrak{V},k^{\prime },%
\mathfrak{D}^{\prime }),$ $(\mathfrak{W},k^{^{\prime \prime }},\mathfrak{D}%
^{^{\prime \prime }})$ be the s.s.embedding triples associated to $(X,%
\mathcal{A})$, $(Y,\mathcal{B})$, and $(Z,\mathcal{C})$ and $b:\mathfrak{U}%
\rightarrow X$, $b^{\prime }:\mathfrak{V}\rightarrow Y$, $b^{\prime \prime }:%
\mathfrak{W}\rightarrow Z$ the morphisms given by the inclusions. By Lemma %
\ref{Lemma_lifting_morph} \textbf{2 \ }$\mathcal{F}|\mathfrak{U}$, $\mathcal{%
G}|\mathfrak{V}$, $\mathcal{H}|\mathfrak{W}$ satisfy the hypothesis of
property \textbf{3-ss }and hence the following diagram commutes:%\bigskip 
\begin{equation}
\begin{tikzcd}[column sep=-0.2cm]
H_{\sharp}\Dolb{k\,}{\mathcal{F}|\mathfrak{U}} &&
G_{\sharp}\Dolb{k^{\prime}\,}{\mathcal{G}|\mathfrak{V}} \arrow{ll} \\
& \Dolb{k^{\prime\prime}\,}{\mathcal{H}|\mathfrak{W}} \arrow{ul}
\arrow{ur} \end{tikzcd}
\end{equation}
To get the result apply $b_{\sharp }^{\prime \prime }$ to the above diagram
and use Lemma \ref{Lemma_Comp_im_dir}.

\begin{remark}
By properties \textbf{2} and \textbf{3} of Theorem \ref{Theor_Dolb} the
functor $\Dolbn$ extends to the category of s.s.locally embedded
analytic spaces, with the pullback morphisms defined using $\sharp $-direct
images (similar to morphism (\ref{morph_pullback_sharp})). In particular, with
the notations at the beginning of paragraph \ref{Paragr_loc_emb_sp}, the
pullback morphism over the natural mapping $b:\mathfrak{U} \rightarrow X$,%
\begin{equation*}
b^{\sharp }:\DolbA{A}{F} \rightarrow b_{\sharp }%
\Dolb{k\,}{\mathcal{F}|\mathfrak{U}}
\end{equation*}%
is the identity of $\DolbA{A}{F}$.
\end{remark}

\section{Further Results and Applications}

\refstepcounter{subsection} \textbf{\arabic{section}.%
\arabic{subsection} }\label{paragr_uniqness_deriv_categ}\textbf{The functor }%
$\Dolbn$ \textbf{in the derived category}. To simplify notation in
what follows we shall omit to write the localization functors such as 
\begin{equation*}
Q:K(\OX{X})\rightarrow D(\OX{X})
\end{equation*}%
It should be clear from the context to which category each complex belongs.

\begin{corollary}
\label{Corol_q_inv}Let $(X,\mathcal{A})$ be a locally embedded analytic
space and let \linebreak $\mathcal{F}\in Mod(\OX{X})$.

\begin{enumerate}
\item The natural functor $\mathcal{F}\rightarrow \DolbA{A}{F}$ gives a functorial isomorphism in the derived category $D^{+}(\OX{X})$

\item Let $S(X)$ be the full subcategory of $D^{+}(\OX{X})$ consisting of complexes with soft terms and 
\begin{equation*}
j:S(X)\rightarrow D^{+}(\OX{X})
\end{equation*}%
the inclusion functor. Then the extension of $\Dolb{\mathcal{A}\,}
{\bullet }$ to $D^{+}(\OX{X})$ is a quasi-inverse for $j$.
\end{enumerate}
\end{corollary}

\begin{proof}
\textbf{1}. follows from Theorem \ref{Theor_Dolb} \textbf{1.b.} while 
\textbf{2. }folows from \textbf{1. }and Theorem \ref{Theor_Dolb} \textbf{1.c.}
\end{proof}

\begin{remark}
Corollary \ref{Corol_q_inv} implies in particular that using $\Dolb{\mathcal{A}\,}{\bullet }$ one defines derived functors for any
functor $F:Mod(\OX{X})\rightarrow Mod(\OX{X})$ s.t. soft sheaves are $F$-acyclic.
\end{remark}

One checks immediately:

\begin{corollary}
Let $(X,\mathcal{A})$ be a locally embedded analytic space and \linebreak $%
\mathcal{F}\in Mod(\OX{X})$. Then the complex $\Gamma (X,\DolbA{A}{F})$ is a representative of $R\Gamma (X,\mathcal{F})$, and hence it computes the cohomology groups $H^{\bullet }(X,\mathcal{F})$.
\end{corollary}

\begin{corollary}
\label{Corol_topol_FS}Let $(X,\mathcal{A})$ be a locally embedded analytic
space and \linebreak $\mathcal{F}\in Coh(\OX{X})$. Assume that $X$
is countable at infinity and that $\mathcal{A}$\ has at most countably many
charts. Then the terms of the complex $\Gamma (X,\DolbA{A}{F})$ have natural topologies of type FS and the differentials are continuous. Furthermore, the terms of $\Gamma (X,\DolbA{A}{F})$ induce the natural topology on the cohomology groups $H^{\bullet }(X,\mathcal{F})$.
\end{corollary}

\begin{proof}
It is well known that $\Gamma (X,\mathcal{E}_{X}^{p,q}\otimes \mathcal{F})$
has a natural topology of type FS (see e.g. \cite{B-S} 7\S 4.b). Thus the global 
sections of the terms in (\ref{Complex_Dolb_A}) are countable products of FS spaces and hence are themselves FS. Note that if $\mathcal{B}$ is another embedding atlas of $X$ then, by Remark \ref{Rem_q_izo_over_id}, $\Gamma (X,\DolbA{A}{F})$ and $\Gamma(X,\DolbA{B}{F})$ induce the same topology on $H^{\bullet }(X,\mathcal{F})$
(since $f^{\ast }(id)$ determines a continuous quasi-isomorphism). Thus, to
check that this topology coincides with the natural one, we can assume that $%
cov(\mathcal{A})$ is a Stein covering. The morphism $v$ in diagram (\ref%
{morph_augment}) determines a continuous quasi-isomorphism on the global
sections:%
\begin{equation*}
C^{\bullet }(cov(\mathcal{A}),\mathcal{F})\rightarrow \Gamma (X,\DolbA{A}{F})
\end{equation*}%
which ends the proof, since the left-hand side complex (the \v{C}ech complex
with respect to $cov(\mathcal{A})$ ) defines the natural topology on 
$H^{\bullet }(X,\mathcal{F})$.
\end{proof}

\begin{remark}
By \textit{\cite{Colt}} and \cite{Colt-Joi} if $X$ is a finite dimensional
analytic space countable at infinity then it can be covered by finitely many
Stein open sets; if moreover $X$ is connected then the Stein open sets can
also be chosen connected. Hence if $X$ has also finite embedding dimension
then it\ has an embedding atlas $\mathcal{A}$ with finitely many charts,
respectively finitely many connected charts, and for any $\mathcal{F}\in Mod(\OX{X})$ the terms of the complex $\DolbA{A}{F}$ consist of products with finitely many factors (see (\ref {Complex_Dolb_A})).
\end{remark}

\begin{theorem}
\label{Theor_Dolb_cat_deriv}Let $f:X\rightarrow Y$ be a morphism of analytic
spaces, $\mathcal{F}\in Mod(\OX{X})$,
\linebreak
$\mathcal{G}\in Mod(\OX{Y})$ and 
$u:\mathcal{G}\rightarrow f_{\ast }\mathcal{F}$ a morphism of $\OX{Y}$-modules. 
Let moreover $\mathcal{A}=(U_{i},k_{i},D_{i})_{i\in I}$ and 
$\mathcal{B}=(V_{j},k_{j}^{\prime },D_{j}^{\prime })_{j\in J}$ be embedding 
atlases of $X$, respectively $Y$. Then

\begin{enumerate}
\item $f_{\ast }\DolbA{A}{F}$ is a representative for $Rf_{\ast }\mathcal{F}$

\item There exists a unique morphism $f^{\ast }(u)$ in $D^{+}(\OX{Y})$ such that the following diagram commutes (in $D^{+}(\OX{Y})$):
\begin{equation}
\begin{CD} 
\DolbA{B}{G} @>f^{\ast }(u)>> f_{\ast }\DolbA{A}{F} \\ 
@AAb^{\prime}A @AAf_{*}bA\\
\mathcal{G} @>{u}>> f_{\ast }\mathcal{F} 
\end{CD}  \label{diagr_f*_u_deriv}
\end{equation}

\item The morphism $f^{\ast }(u)$ can be represented as a sequence of
pullback morphisms

\item Let $g:Y \rightarrow Z$ be another morphism of analytic spaces and $%
h=g\circ f$. Let moreover $\mathcal{H}\in Mod(\OX{Z})$ and $v:%
\mathcal{H} \rightarrow g_{\ast }\mathcal{G}$, $w:\mathcal{H} \rightarrow %
h_{\ast }\mathcal{F}$ morphisms of $\OX{Z}$-modules, such that $%
g_{\ast }(u)\circ v=w$. If $\mathcal{C}$\ is an embedding atlas of $Z$ then,
in the derived category $D^{+}(\OX{Z})$, one has the commutative
diagram: 
\begin{equation}
\begin{tikzcd}[column sep=-0.2cm] h_{\ast}\DolbA{A}{F} && g_{\ast}\DolbA{B}{G}
\arrow{ll}[swap]{Rg_{\ast}(f^{\ast}(u)} \\ & \DolbA{C}{H} \arrow{ul}{h^{\ast}(w)} \arrow{ur}[swap]{g^{\ast}(v)}
\end{tikzcd}
\end{equation}
\end{enumerate}
\end{theorem}

\begin{proof}
\textbf{1.} and \textbf{2.} are obvious.

\textbf{3. }Assume first that the embedding atlases $\mathcal{A}$, $\mathcal{%
B}$ are $f$-compliant. If $\tau $ is a refinement mapping, consider the
diagram of locally embedded analytic spaces (see Lemma \ref%
{Lemma_complet_morph_atlas} ): 
\begin{equation}
(X,\mathcal{A})\overset{P_{1}}{\longleftarrow }(X,\mathcal{A}\times _{\tau }%
\mathcal{B})\overset{P_{2}}{\longrightarrow }(Y,\mathcal{B})
\label{Diagr_tau}
\end{equation}%
and let 
\begin{equation}
f_{\ast }\DolbA{A}{F}\overset{f_{\ast }P_{1}^{\ast }(id)}
{\longrightarrow }f_{\ast }\Dolb{\mathcal{A}\times _{\tau }\mathcal{B}\,}{\mathcal{F}}\overset{P_{2}^{\ast }(u)}{\longleftarrow }\DolbA{B}{G}
\label{Diagr_Def_f*_deriv}
\end{equation}%
be the corresponding $\Dolbn$-diagram (i.e the diagram of pullback
morphisms over the arrows in diagram (\ref{Diagr_tau})),\ where 
\begin{equation*}
P_{1}=(id,id,p_{1})\text{ and }P_{2}=(f,\tau ,p_{2})
\end{equation*}%
Note that since the components of the $\Dolbn $-resolutions are soft
sheaves, Remark \ref{Rem_q_izo_over_id}\ implies that $f_{\ast }P_{1}^{\ast
}(id)$ is a quasi-isomorphism, and so diagram (\ref{Diagr_Def_f*_deriv})
gives a morphism in $D^{+}(\OX{Y})$ which coincides with $f^{\ast
}(u)$ (to see this use diagrams similar to (\ref{diagr_f*_u_deriv}) for the
morphisms in diagram (\ref{Diagr_Def_f*_deriv}))

\begin{remark}
The morphism given by diagram (\ref{Diagr_Def_f*_deriv}) does not depend on
the refinement mapping $\tau $(use, for instance, the $\Dolbn $-diagram over diagram 
(\ref{Diagr_tau12}) in Lemma \ref{Lemma_complet_morph_atlas}). Moreover let 
$\mathcal{A}^{\prime }$ be an embedding atlas on $X$ such that one has the diagram of 
locally embedded analytic spaces:%
\begin{equation*}
(X,\mathcal{A})\overset{(id,\upsilon ,q_{1})}{\longleftarrow }(X,\mathcal{A}%
^{\prime })\overset{(f,\tau \circ \upsilon ,q_{2})}{\longrightarrow }(Y,%
\mathcal{B})
\end{equation*}%
Then the corresponding $\Dolbn $-diagram is also a representative for 
$f^{\ast }(u)$ (use, for instance, the $\Dolbn $-diagram over
diagram (\ref{Diagr_A_prim}) in Lemma \ref{Lemma_complet_morph_atlas})\ 
\end{remark}

\begin{remark}
\label{Rem_izo_cat_deriv}Assume that $Y=X$, $f=id_{X}$, $\mathcal{G}=%
\mathcal{F}$, and $u=id$. Remark\nolinebreak\ \ref{Rem_q_izo_over_id}
implies $f^{\ast }(u)$ is an isomorphism in the derived category.
\end{remark}

Now drop the suplimentary assumption. Let $\mathcal{A}^{\prime }$ be an
embedding atlas of $X$ s.t. $\mathcal{A}^{\prime }$, $\mathcal{A}$ are $%
id_{X}$-compliant and $\mathcal{A}^{\prime }$, $\mathcal{B}$ are $f$%
-compliant (choose for instance an embedding atlas over the open covering 
$cov(\mathcal{A})\cap f^{-1}cov(\mathcal{B}))$. Then the $\Dolbn $%
-diagram in $D^{+}(\OX{Y})$
\begin{equation*}
f_{\ast }\DolbA{A}{F} \overset{f_{\ast }id^{\ast }(id)}
{\longrightarrow }f_{\ast }\Dolb{\mathcal{A}^{\prime }}
{\mathcal{F}}\overset{f^{\ast }(u)}{\longleftarrow }\DolbA{B}{G}
\end{equation*}%
determines a morphism:%
\begin{equation*}
\DolbA{B}{G} \rightarrow f_{\ast }\DolbA{A}{F}
\end{equation*}%
since, by Remark \ref{Rem_izo_cat_deriv}, the left-hand arrow is an
isomorphism. Moreover one checks, as in the f-compliant case, that this
morphism coincides with $f^{\ast }(u)$.

\textbf{4. }follows from the equality $g_{\ast }(u)\circ v=w$ by considering
the diagrams similar to diagram (\ref{diagr_f*_u_deriv}) for each morphism.
Alternatively, in the compliant case (i.e. if $\mathcal{A}$, $\mathcal{B}$
are $f$-compliant and $\mathcal{B}$, $\mathcal{C}$ are $g$-compliant) the
claim follows from the $\mathcal{D}olb$-diagram over diagram (\ref%
{Diagr_comut_morph_compoz})\ in Remark \ref{Rem_comut_morph_compoz}. The
general case reduces to the compliant one via isomorphisms.
\end{proof}

\refstepcounter{subsection}\textbf{\ \arabic{section}.\arabic{subsection} }%
\label{paragr_Reduced_space} \textbf{Dolbeault resolutions on reduced analytic 
spaces.} Let $X$ be a reduced analytic space. Using an
embedding atlas of a particular form one can construct a
Dolbeault-Grothendieck type resolution which coincides with the usual one on
the regular part of $X$. For this let $\mathcal{F}\in Mod(\OX{X})$,
let $(W_{j})_{j}$ be a neighbourhood basis of $Sing(X)$, and $%
(U_{i},k_{i},D_{i})_{i\in I}$ a family of embedding triples s.t. $%
(U_{i})_{i\in I}$ are open sets of $X$ which cover $Sing(X)$. Denote 
\begin{equation*}
S=Sing(X)\text{ and }U_{0}=Reg(X)
\end{equation*}%
The following Lemma is obvious:

\begin{lemma}
\begin{enumerate}
\item The family of charts $(U_{i},k_{i},D_{i})_{i\in I}$ together with the
embedding triple $(U_{0},id,U_{0})$ give an embedding atlas $\mathcal{A}$ of 
$X$.

\item For each $j\in J$, the family of embedding triples $(U_{i}\cap
W_{j},k_{i}|W_{j},D_{i}^{\prime })_{i\in I}$, where $D_{i}^{\prime }\subset
D_{i}$ is a suitable open subset, together with the embedding triple $%
(U_{0},id,U_{0})$ give an embedding atlas $\mathcal{A}^{(j)}$ of $X$.

\item $\Dolb{\mathcal{A}^{(j)}}{\mathcal{F}}$ does not depend on the
choice of the open subsets $D_{i}^{\prime }$.

\item If $W_{j_{1}}\subset W_{j_{2}}$ then there is a natural pullback
morphism \linebreak $i_{j_{1}j_{2}}^{\ast }(id):\Dolb{\mathcal{A}%
^{(j_{2})}}{\mathcal{F}}\rightarrow \Dolb{\mathcal{A}^{(j_{1})}}{\mathcal{F}}$
over the identity of $X$, and $(\Dolb{\mathcal{A}^{(j)}}{\mathcal{F}}%
)_{j}$ is an inductive system of complexes of $\OX{X}$-modules.
\end{enumerate}
\end{lemma}

We set:%
\begin{equation}
r\DolbA{A}{F} = \lim\limits_{\overrightarrow{j}}%
\Dolb{\mathcal{A}^{(j)}}{\mathcal{F}}  \label{Def_Dolb_redus}
\end{equation}%
It is easy to see that the definition of $r\DolbA{A}{F}$ is independent of the neighbourhood basis $(W_{j})_{j}$.

\begin{corollary}
$r\Dolb{\mathcal{A}\,}{\bullet }$ is a functor $Mod(\OX{X})\rightarrow C^{+}(X)$. 
Moreover properties 
\linebreak
\textbf{1.a} - \textbf{1.e} in Theorem \ref{Theor_Dolb} hold 
for $\Dolb{\mathcal{A}\,}{\bullet }$ replaced by  $r\Dolb{\mathcal{A}\,}{\bullet }$.
\end{corollary}

\begin{proof}
\textbf{1.a}, \textbf{1.b}, \textbf{1.d}, \textbf{1.e} follow immediately
from the respective properties in Theorem \ref{Theor_Dolb} because of the
compatibility with inductive limits. For \textbf{1.c }note that the terms of 
$r\DolbA{A}{F}$ consist of Godement restrictions to the closed set $S$ of soft sheaves, and consequently are also soft.
\end{proof}

\begin{corollary}
$r\Dolb{\mathcal{A}\,}{\OX{X}}|U_{0}$ coincides with the Dolbeault-Grothendieck resolution on the manifold $U_{0}$.
\end{corollary}

\begin{remark}
Assume $\mathcal{F}\in Coh(\OX{X})$. Then the topologies on the
global sections spaces $\Gamma (X,r\Dolb{\mathcal{A}\,}{\OX{X}})$ are more complicated than the FS topologies of Corollary \ref{Corol_topol_FS}. However, since the natural quasi-isomorphism 
\begin{equation}
\Gamma (X,\DolbA{A}{F})\rightarrow \Gamma (X,r\DolbA{A}{F})
\end{equation}%
is continuous, both complexes induce the same topology on the cohomology
groups $H^{\bullet }(X,\mathcal{F})$.
\end{remark}

\begin{remark}
If $f:(X,\mathcal{A})\rightarrow (Y,\mathcal{B})$ is a morphism of locally embedded
analytic spaces, $\mathcal{G}\in Mod(\OX{Y})$ and $u:\mathcal{%
G}\rightarrow f_{\ast }\mathcal{F}$ a morphism of $\OX{Y}$-modules
then one has a pullback morphism 
\begin{equation}
f^{\ast }(u):\DolbA{B}{G}\rightarrow f_{\ast }r\DolbA{A}{F}
\end{equation}%
but, in general, not a morphism 
\begin{equation}
r\DolbA{B}{G} \rightarrow f_{\ast }r\DolbA{A}{F}
\end{equation}%
when $Y$ is also a reduced analytic space.
\end{remark}

\refstepcounter{subsection} \textbf{\arabic{section}.\arabic{subsection}} %
\label{paragr_deRham} \textbf{The functor }$\Dolbn $\textbf{\ and the
de Rham complex on analytic spaces}. Let $X$ be an analytic space and%
\begin{equation}
0\longrightarrow \Omega_{X}^{0}\overset{{\partial }_{X}^{0i}}{%
\longrightarrow }\Omega_{X}^{1}\overset{{\partial }_{X}^{1}}{%
\longrightarrow }...
\end{equation}%
be the de Rham complex on $X$ (see e.g. H.Grauert, H.Kerner \cite{Grauert}\
or A.Grothendieck \cite{Groth}). Recall that if $k:X\hookrightarrow D$ is a
closed embedding of $X$ in the complex manifold $D$, and $\mathcal{I}%
_{X}\subset \OX{D}$ is the coherent ideal sheaf which gives $X$ as
a subspace of $D$, then 
\begin{equation}
\Omega _{X}^{i}=\Omega _{D}^{i}/\mathcal{N}_{X}^{i}
\end{equation}%
where $\mathcal{N}_{X}^{i}$ is the $\OX{D}$-submodule of $\Omega
_{D}^{i}$ generated by $\mathcal{I}_{X}\Omega _{X}^{i}$ and ${%
\partial }_{D}^{i-1}(\mathcal{I}_{X}\Omega _{X}^{i-1})$, and the
differentials 
\begin{equation}
\partial _{X}^{i}:\Omega _{X}^{i}\rightarrow \Omega _{X}^{i+1}
\end{equation}%
are induced by those of $\Omega _{D}^{\bullet }$. To simplify notation in
what follows we shall write $\partial $ instead of ${%
\partial }_{X}^{i}$ if $i$ and $X$ are clear from the context.

If $f:X\rightarrow Y$ is a morphism of analytic spaces, one has a pullback
morphism: 
\begin{equation}
f^{\ast }:\Omega _{Y}^{\bullet }\rightarrow f_{\ast }\Omega _{X}^{\bullet }
\label{Morph_pullback_an_sp}
\end{equation}%
In particular, if $f:(X,i_{1},D_{1})\rightarrow (Y,i_{2},D_{2})$ is a
morphism of embedding triples then the morphism (\ref{Morph_pullback_an_sp})
is induced by the usual pullback morphism 
\begin{equation}
\Omega _{D_{2}}^{\bullet }\rightarrow f_{\ast }\Omega _{D_{1}}^{\bullet }
\end{equation}%
since one checks that $f^{\ast }(\mathcal{N}_{Y}^{i})\subset \mathcal{N}%
_{X}^{i}$ for all $i$.

\begin{remark}
\label{Rem_ss_de_Rham}If $\mathscr{X}=(X_{\alpha })_{\alpha }$ is a
s.s.analytic space then the functoriality of the pullback morphisms (\ref%
{Morph_pullback_an_sp}) implies that $\Omega _{\mathscr{X}}^{\bullet
}=(\Omega _{X_{\alpha }}^{\bullet },\partial )_{\alpha }$ is a complex of $%
\mathscr{X}$-modules with $\mathbb{C}$-linear differentials.
\end{remark}

\begin{theorem}
\label{Theor_Dolb_deRham}Let $(X,\mathcal{A})$ be a locally embedded
analytic space.

\begin{enumerate}
\item The differential ${\partial }_{X}^{i}:\Omega
_{X}^{i}\rightarrow \Omega _{X}^{i+1}$ induces a $\mathbb{C}$-linear
morphism of resolutions:%
\begin{equation}
{\partial }^{i}:\Dolb{\mathcal{A}\,}{\Omega_{X}^{i}}\rightarrow \Dolb{\mathcal{A}\,}{\Omega _{X}^{i+1}}
\label{Morpf_Dolb_de_Rham_A}
\end{equation}%
such that $\Dolb{\mathcal{A}\,}{\Omega _{X}^{\bullet }}$ is a double
complex

\item The simple complex associated to $\Dolb{\mathcal{A}\,}{\Omega
_{X}^{\bullet }}$ is a resolution of $\Omega _{X}^{\bullet }$ with soft, 
$\OX{X}$-flat sheaves
\end{enumerate}
\end{theorem}

\begin{proof}
\textbf{1. }The morphism of resolutions is obtained by following the
construction of the functor $\Dolb{\mathcal{A}\,}{\bullet }$ in
Section \ref{Sect_Resol}.

\textbf{a. (Smooth case)} Let $D$ be a complex manifold and consider the
morphism of resolutions: 
\begin{equation}
\begin{tikzcd}[column sep=1.5cm] 0 \arrow{r} & \Omega^{i+1}_{D} \arrow{r} &
\mathcal{E}_{D}^{i+1,0} \arrow{r} & \mathcal{E}_{D}^{i+1,1} \arrow{r} &
\ldots \\ 0 \arrow{r} & \Omega^{i}_{D} \arrow{r} \arrow{u}{\partial } &
\mathcal{E}_{D}^{i,0} \arrow{r} \arrow{u}{\partial } & \mathcal{E}_{D}^{i,1}
\arrow{r} \arrow{u}{\partial } & \ldots \end{tikzcd}
\end{equation}%
Using the natural isomorphisms:%
\begin{equation}
\mathcal{E}_{D}^{0,j}\otimes _{\OX{D}}\Omega _{D}^{i}\widetilde{%
\rightarrow }\mathcal{E}_{D}^{i,j}
\end{equation}%
one gets a morphism of resolutions: $\Dolb{D\,}{\Omega _{D}^{i}}\rightarrow \Dolb{D\,}{\Omega _{D}^{i+1}}$. Obviously
\linebreak
${\partial }^{i+1}\circ {\partial }^{i}=0$ and $\Dolb{D\,}{\Omega _{D}^{\bullet }}$ 
is a double complex.

Note that if $(z_{1},\ldots ,z_{n})$ are local coordinates on $D$ then the
morphism 
\begin{equation}
\partial ^{i}:\mathcal{E}_{D}^{0,j}\otimes _{\OX{D}}\Omega
_{D}^{i}\rightarrow \mathcal{E}_{D}^{0,j}\otimes _{\OX{D}}\Omega
_{D}^{i+1}
\end{equation}%
is given by 
\begin{equation}
\partial ^{i}(\alpha \otimes \omega )=\sum_{k=1}^{n}\frac{\partial }{%
\partial z_{k}}(\alpha )\otimes dz_{k}\wedge \omega +(-1)^{j}\alpha \otimes
\partial \omega
\end{equation}

\textbf{b. (Embedded case)} If $(X,k,D)$ is an embedding triple then one
checks that 
\begin{equation}
\partial (\mathcal{E}_{D}^{0,j}\otimes _{\OX{D}}\mathcal{N}%
_{X}^{i})\subset \mathcal{E}_{D}^{0,j}\otimes _{\OX{D}}\mathcal{N}%
_{X}^{i+1}
\end{equation}%
and consequently one obtains a morphism of resolutions:%
\begin{equation}
\partial :\Dolb{k\,}{\Omega _{X}^{i}} \rightarrow \Dolb{k\,}
{\Omega _{X}^{i+1}}  \label{Morph_Dolb_de_Rham}
\end{equation}%
and $\Dolb{k\,}{\Omega _{X}^{\bullet }}$ becomes a double complex. 
Moreover the differentials (\ref{Morph_Dolb_de_Rham}) are compatible with
the pullback morphisms.

\textbf{c. (General case)} Let $(\mathfrak{U},k,\mathfrak{D})$ be the
s.s.embedding triple associated to $(X,\mathcal{A})$ and $b:\mathfrak{U}%
\rightarrow X$ the natural morphism given by the inclusions. The morphisms (%
\ref{Morph_Dolb_de_Rham}) give a morphism of resolutions 
\begin{equation}
\Dolb{k\,}{\Omega _{X}^{i}|\mathfrak{U}}\rightarrow \Dolb{k\,}{\Omega_{X}^{i+1}|\mathfrak{U}}
\end{equation}
and by applying $b_{\sharp }$ the morphism (\ref{Morpf_Dolb_de_Rham_A}). It
is immediate to check that $\Dolb{\mathcal{A}\,}{\Omega _{X}^{\bullet}}$ is a double complex.

\textbf{2. }is obvious.
\end{proof}

\refstepcounter{subsection} \textbf{\arabic{section}.\arabic{subsection}}%
\label{Paragr_smooth_forms} \textbf{The functor }$\Dolbn$ \textbf{%
and the complex of smooth differential forms}

Let $X$ be an analytic space and let%
\begin{equation}
0\longrightarrow \mathcal{E}_{X}^{0,0}\overset{\overline{\partial }_{X}^{0}}{%
\longrightarrow }\mathcal{E}_{X}^{0,1}\overset{\overline{\partial }_{X}^{1}}{%
\longrightarrow }...
\end{equation}%
be the complex of smooth differential forms with first degree $0$ on $X$.
Recall that if $k:X\hookrightarrow D$ is a closed embedding of $X$ in the
complex manifold $D$, and $\mathcal{I}_{X}\subset \OX{D}$ is the
coherent ideal sheaf which gives $X$ as a subspace of $D$, then 
\begin{equation}
\mathcal{E}_{X}^{0,i}=\mathcal{E}_{D}^{0,i}/\mathcal{M}_{X}^{i}|X
\end{equation}%
where $\mathcal{M}_{X}^{i}$ is the $\OX{D}$-submodule of $\mathcal{E%
}_{D}^{0,i}$ generated by $\mathcal{I}_{X}\mathcal{E}_{D}^{0,i}$, $\overline{%
\mathcal{I}}_{X}\mathcal{E}_{D}^{0,i}$, and \linebreak $\overline{%
\partial }^{i-1}(\overline{\mathcal{I}}_{X}\mathcal{E}_{D}^{0,i-1})$; the
differentials 
\begin{equation}
\overline{\partial }_{X}^{i}:\mathcal{E}_{X}^{0,i}\rightarrow 
\mathcal{E}_{X}^{0,i+1}
\end{equation}%
are induced by those of $\mathcal{E}_{D}^{0,\bullet }$. If $X$ is a reduced
analytic space $X$ then $\mathcal{M}_{X}^{i}$ consists of the forms in $%
\mathcal{E}_{D}^{0,i}$ which have null pullback to $Reg(X)$. To simplify
notation, in what follows we shall write $\overline{\partial }$
instead of $\overline{\partial }_{X}^{i}$ if $i$ and $X$ are clear
from the context.

Note that the natural surjective morphisms%
\begin{equation}
\mathcal{E}_{D}^{0,i}/\mathcal{I}_{X}\mathcal{E}_{D}^{0,i}|X 
\rightarrow \mathcal{E}_{X}^{0,i}
\end{equation}%
determine a natural morphism of complexes:%
\begin{equation}
\Dolb{k\,}{\OX{X}}\rightarrow \mathcal{E}_{X}^{0,\bullet }
\end{equation}%
In general one proves:

\begin{theorem}
\label{Theor_Dolb_diff_forms}Let $(X,\mathcal{A})$ be a locally embedded
analytic space. Then there is a surjective morphism of complexes of sheaves%
\begin{equation}
\Dolb{\mathcal{A}\,}{\OX{X}}\rightarrow \mathcal{C}^{\bullet
}(\mathcal{U},\mathcal{E}_{X}^{0,\bullet })  \label{Morph_E_X}
\end{equation}%
where $\mathcal{U} = cov(\mathcal{A})$ is the open covering of $X$ 
corresponding to the atlas $\mathcal{A}$ and $\mathcal{C}^{\bullet }(%
\mathcal{U},\bullet )$ denotes the \v{C}ech complex on $\mathcal{U}$.
\end{theorem}

\begin{proof}
Let $(\mathfrak{U,}k,\mathfrak{D})$ be the s.s.embedding triple associated
to $(X,\mathcal{A})$ and $b:\mathfrak{U}\rightarrow X$ the natural morphism
given by the inclusions. For each $\alpha \in \mathcal{N(U)}$ one has a
morphism 
\begin{equation}
\Dolb{k_{\alpha }\,}{\OX{U\alpha }} \rightarrow \mathcal{E}_{U\alpha }^{0,\bullet }
\end{equation}%
and these morphisms give a morphism%
\begin{equation}
\Dolb{k\,}{\mathcal{O}|\mathfrak{U}})\rightarrow \mathcal{E}%
_{X}^{0,\bullet }|\mathfrak{U}  \label{Morph_E_U}
\end{equation}%
The morphism (\ref{Morph_E_X}) is obtained by applying $b_{\sharp }$ to (\ref%
{Morph_E_U}).
\end{proof}

\end{document}